\documentclass[a4, 12pt, 
]{amsart}

\usepackage[dvipdfmx]{graphicx,xcolor}
\usepackage[dvipdfmx]{hyperref}
\usepackage{amsmath}
\usepackage{amssymb}
\usepackage[dvipdfmx]{hyperref}
\usepackage{hyperref}
\hypersetup{
	colorlinks=true,
	linkcolor=red,
	citecolor=blue}
	\usepackage[utf8]{inputenc}
\usepackage[T1]{fontenc}
\usepackage{color}
\usepackage{xcolor} 
\usepackage{ulem}
\usepackage{calligra}
\usepackage{mathrsfs}
\usepackage[all]{xy}

\usepackage{comment}
\usepackage{todonotes}

\newtheorem{theo}{Theorem}[section]

\newtheorem{lemm}[theo]{Lemma}
\newtheorem{cor}[theo]{Corollary}
\newtheorem{claim}[theo]{Claim}

\numberwithin{equation}{section}

\theoremstyle{definition}
\newtheorem{defi}[theo]{Definition}

\newtheorem{setup}[theo]{Setting}
\newtheorem{step}{Step}

\theoremstyle{remark}
\newtheorem{rem}[theo]{Remark}

\newcommand{\Ker}[0]{\operatorname{Ker}}

\newcommand{\rank}[0]{\operatorname{rank}}

\newcommand{\Nak}{\mathrm{Nak.}}

\newcommand{\Id}{\mathrm{Id}}
\newcommand{\st}{\mathrm{st}}

\newcommand{\deldel}{\sqrt{-1}\partial \overline{\partial}}
\newcommand{\Dbar }{\overline{\partial}}
\newcommand{\e}{\varepsilon}
\newcommand{\ai}{\sqrt{-1}}
\newcommand{\I}{\mathcal{I}}

\newcommand{\C}{\mathbb{C}}

\newcommand{\glo}{\mathrm{glo.}}
\newcommand{\loc}{\mathrm{loc.}}

\newcommand{\cal}[1]{\mathcal{#1}}

\newcommand{\U}[1]{{U}_{i_{0}...i_{#1}}}

\newcommand{\idd}{\sqrt{-1}\partial\overline{\partial}}

\usepackage{geometry}
\begin{document}

\title[Singular Nakano positivity of direct image sheaves]
{Singular Nakano positivity \\ of direct image sheaves of adjoint bundles}

\author{Takahiro INAYAMA}
\address{Department of Mathematics, 
Faculty of Science and Technology, 
Tokyo University of Science, 
2641 Yamazaki, Noda, 
Chiba, 278-8510, 
Japan.
}
\email{{\tt inayama\_takahiro@rs.tus.ac.jp}, {\tt inayama570@gmail.com}} %(T.\,INAYAMA)}
%\email{{\tt inayama570@gmail.com}}

\author{Shin-ichi MATSUMURA}
\address{Mathematical Institute, Tohoku University, 
6-3, Aramaki Aza-Aoba, Aoba-ku, Sendai 980-8578, Japan.}
\email{{\tt mshinichi-math@tohoku.ac.jp}, {\tt mshinichi0@gmail.com}} %(S.\,MATSUMURA)}
%\email{{\tt mshinichi0@gmail.com}}

\author{Yuta WATANABE$^\ast$}
\address{Department of Mathematics, Faculty of Science and Engineering, Chuo University.
1-13-27 Kasuga, Bunkyo-ku, Tokyo 112-8551, Japan}
\email{{\tt wyuta@math.chuo-u.ac.jp}, {\tt wyuta.math@gamil.com}} %(Y.\,WATANABE)}
%\email{{\tt wyuta.math@gamil.com}}
%\thanks{$^\ast$\,Corresponding author} %: Yuta WATANABE, wyuta.math@gamil.com}

\date{\today, version 0.02}

\renewcommand{\subjclassname}{%
\textup{2020} Mathematics Subject Classification}


\maketitle

\begin{abstract}
In this paper, we consider a proper K\"ahler fibration $f \colon X \to Y$ 
and a singular Hermitian line bundle $(L, h)$ on $X$ with semi-positive curvature. 
We prove that the canonical $L^2$-metric on the direct image sheaf $f_{*}(\mathcal{O}_{X}(K_{X/Y}+L) \otimes \I(h))$ is singular Nakano semi-positive in the sense that the $\Dbar$-equation can be solved with optimal $L^{2}$-estimate. 
Our proof does not rely on the theory of Griffiths positivity for the direct image sheaf.
\end{abstract}

\tableofcontents

\section{Introduction}\label{Sec-1}

The theory of positivity of direct image sheaves of adjoint bundles originates 
from Griffiths' work  \cite{Gri70}, 
which is based on the theory of variation of Hodge structures.
This theory has been further developed in subsequent works \cite{Fuj78, Kaw81, Vie83} 
and other subsequent contributions from the viewpoint of algebraic geometry.
On the other hand, the analytic theory of the positivity of direct image sheaves 
was initiated by the theory of variation of Bergman kernels by \cite{MY04} and further advanced in \cite{Ber09}.
This analytic approach has been extended by \cite{HPS18, PT18} 
to establish the theory of \textit{singular Griffiths positivity} of direct image sheaves 
as a more applicable framework.
The main result of this paper (see Theorem \ref{thm-main} below)
establishes \textit{singular Nakano positivity} of direct image sheaves
in the same framework as in \cite{HPS18, PT18}, and further generalizes 
the previous works  \cite[Theorem 1.2]{Wat23} and \cite[Theorem 1.5]{DNWZ23} 
on singular Nakano positivity. 
Here, we adopt the definition of singular Nakano positivity based on \cite{DNWZ23, Ina22},
which uses solutions of the $\Dbar$-equation with optimal $L^{2}$-estimates of H\"ormander type
(see Definition \ref{def-dbarNak} for the precise definition).

\begin{theo}\label{thm-main}
Let $f \colon X \to Y$ be a proper K\"ahler fibration between complex manifolds $X$ and $Y$ 
$($i.e.,\,a proper surjective morphism with connected fibers such that 
the inverse image $f^{-1}(U_{Y})$ is a K\"ahler manifold 
for a sufficiently small open set $U_{Y}\subset Y$$)$. 
Let $(L, h)$ be a singular Hermitian line bundle on $X$ such that 
$\sqrt{-1}\Theta_{h} \geq f^{*}\theta$ holds on $Y$ in the sense of currents, 
where $\theta$ is a continuous $(1,1)$-form on $Y$. 
Then, we have the following$:$
\smallskip

$(1)$
The induced canonical $L^2$-metric $H$ on the direct image sheaf $f_{*}(\mathcal{O}_{X}(K_{X/Y} +L) \otimes \I(h))$ 
satisfies that 
$$
\ai\Theta_H \geq^{L^2}_{\loc \Nak} \theta \otimes \Id \text{ on } Y \text{ in the sense of Definition } \ref{def-dbarNak-sheaf}.
$$

$(2)$ If we further assume that $X$ is K\"ahler, then the canonical $L^2$-metric $H$ satisfies that 
$$
\ai\Theta_H \geq^{L^2}_{\glo \Nak} \theta \otimes \Id \text{ on } Y \text{ in the sense of Definition } \ref{def-dbarNak-sheaf}. 
$$
Here $K_{X/Y}$ is the relative canonical bundle and $\mathcal{I}(h)$ is the multiplier ideal  sheaf 
associated with $h$. 
\end{theo}

To clarify our contribution, we review and compare the previous works to Theorem \ref{thm-main}.
Consider the setting of Theorem \ref{thm-main} and suppose that $\theta=0$.
When $h$ is a smooth Hermitian metric and $f \colon X \to Y$ is a smooth fibration (i.e.,\,a holomorphic submersion),
the breakthrough work \cite{Ber09} proved that 
the canonical $L^2$-metric $H$
is a smooth Hermitian metric on $f_{*}(\mathcal{O}_{X}(K_{X/Y}+L))$
whose Chern curvature is Nakano positive.
This result has been generalized to the ``non-smooth'' setting by \cite{HPS18, PT18}.
More specifically, building upon the work \cite{PT18},
the work \cite{HPS18} proved that $H$ is a (possibly) singular Hermitian metric on 
$f_{*}(\mathcal{O}_{X}(K_{X/Y}+L) \otimes \I(h))$ satisfying Griffiths positivity.
This is an elegant application of the Ohsawa-Takegoshi $L^{2}$-extension theorem with optimal $L^{2}$-estimate.
For smooth Hermitian metrics, Nakano positivity is a stronger notion than Griffiths positivity.  
Therefore, the study of Nakano positivity in the ``non-smooth'' setting is a natural and important problem. 
In tackling this problem,
we face the difficulty that the Chern curvature cannot be defined for singular Hermitian metrics (see \cite{Rau15}), and the correct definition of Nakano positivity is not immediately obvious.
The work \cite{DNWZ23}  showed that Nakano positivity of smooth Hermitian metrics
can be characterized by the $\overline{\partial}$-equation with optimal $L^2$-estimate.
In this paper, by adopting the definition based on this characterization,
we prove the main result, generalizing not only \cite{Ber09} to the ``non-smooth'' setting 
but also \cite{HPS18, PT18} to singular Nakano positivity.
We emphasize that our proof does not rely on the Griffiths positivity.

This paper presents two applications of Theorem \ref{thm-main}. 
The first application proves the singular Nakano positivity of the Narasimhan--Simha metric 
on direct image sheaves of relative pluri-canonical bundles (see Corollary \ref{cor-main}). 
In this proof, as opposed to Theorem \ref{thm-main}, 
we use the Griffiths positivity results established in \cite{HPS18} and \cite{PT18}.

A vanishing theorem of Kodaira type can be expected for singular Nakano positive vector bundles,
which is an advantage of considering Nakano positivity compared to Griffiths positivity.
In this context, as the second application,
we prove a vanishing theorem of Koll\'ar-Ohsawa type
(see Corollary \ref{cor-vanish} and cf.\,\cite{Mat16, Ohs84, Kol86a, Kol86b}).
Corollary \ref{cor-vanish} follows directly from Theorem \ref{thm-main} and \cite[Theorem 1.5]{Ina22} 
when $Y$ is projective and $f_{*}(\mathcal{O}_{X}(K_{X}+L) \otimes \I(h))$ is a locally free sheaf.
However, in the general case,
we utilize the proof of Theorem \ref{thm-main}
to solve the $\overline{\partial}$-equation with optimal $L^2$-estimates,
thereby deducing the vanishing theorem.
In this regard, Corollary \ref{cor-vanish} can be viewed as a quantitative version
of the vanishing theorem of Koll\'ar-Ohsawa type.

\begin{cor}[{{cf.\,\cite{Mat16, Ohs84, Kol86a, Kol86b}}}] \label{cor-vanish}
Consider the same situation as in Theorem \ref{thm-main}. 
Assume that $X$ is a weakly pseudo-convex K\"ahler manifold and  $\sqrt{-1}\Theta_{h} \geq f^{*} \omega_{Y}$ holds on $Y$ 
for some K\"ahler form $\omega_{Y}$ on $Y$. 

Then, we have the following vanishing theorem$:$
$$
H^{q}(Y, f_{*}(\mathcal{O}_{X}(K_{X}+L) \otimes \I(h))) =0 \quad \text{ for all $q>0$.}
$$ 

\end{cor}

\subsection*{Organization of the paper}\label{organization}
The paper is organized as follows:
In Section \ref{Sec-2}, we review the definition of singular Nakano positivity
and prove the technical lemmas necessary for the proof of the main result.
Section \ref{Sec-3} is divided into two parts:
In the first half, we solve the $\bar{\partial}$-equation with optimal $L^{2}$-estimates
in an appropriate setting.
In the latter half, we derive Theorem \ref{thm-main} and Corollary \ref{cor-vanish}.
In Section \ref{Sec-4}, as a corollary of Theorem \ref{thm-main}, 
we prove the singular Nakano positivity of the Narasimhan--Simha metric on direct image sheaves. 
In Appendix \ref{appendix}, we summarize useful results on the approximation and extension of singular Hermitian metrics, which may already be known to experts.

\subsection*{Acknowledgment}\label{subsec-ack}
The authors would like to thank Professor Masataka Iwai and Professor Hisashi Kasuya, the organizers 
of the \lq \lq Workshop on Complex Geometry in Osaka," where their collaboration began.
They would also like to thank Professor Shigeharu Takayama 
for suggesting the addition of Corollary \ref{cor-vanish}. 
T.\,I.\,is supported by Grant-in-Aid for Early-Career Scientists $\sharp$23K12978 
from the Japan Society for the Promotion of Science (JSPS).
S.\,M.\,is supported 
by Grant-in-Aid for Scientific Research (B) $\sharp$21H00976 from JSPS.

\section{Preliminaries}\label{Sec-2}

We begin by reviewing the definition of singular Nakano positivity.
The direct image sheaf treated in this paper is a torsion-free sheaf, but not necessarily locally free.
Therefore, for our purpose, we require Definition \ref{def-dbarNak-sheaf}.

\begin{defi}[Singular Nakano positivity for vector bundles, {\cite{DNWZ23, Ina22}}]\label{def-dbarNak} 
$ $
Let $E$ be a vector bundle on a complex manifold $X$ and 
$\theta$ be a continuous  $(1,1)$-form on $X$. 
Consider a singular Hermitian metric $h$ on  $E$ on $X$ with $h^*$ being upper semi-continuous 
(i.e.,\,the function $|u|^2_{h^*}$ is upper semi-continuous for any local  section $u$ of $E^*$). 
\smallskip

 (1) We define \textit{the global $\theta$-Nakano positivity in the sense of $L^2$-estimates}, denoted by 
    \[
    \ai\Theta_h \geq^{L^2}_{\glo\Nak} \theta \otimes \Id_{E} \text{ on } X,
    \]
as follows:
For any data consisting of: 
\begin{itemize}
\item a Stein coordinate $\Omega \subset X$ admitting a trivialization $E|_\Omega\cong \Omega\times \C^r$; 
\item a K\"ahler form $\omega_\Omega$ on $\Omega$; 
\item a smooth function $\psi$ on $\Omega$ such that $ \theta + \deldel \psi >0$; 
\item a positive integer $q$ with $1\leq q\leq n$; 
\item a $\overline{\partial} $-closed $g\in L^2_{n,q}(\Omega, E; \omega_\Omega, he^{-\psi})$ 
satisfying 
$$\int_\Omega \langle B^{-1}_{\omega_\Omega,\psi,\theta }g, g\rangle_{\omega_\Omega,h}e^{-\psi} \,dV_{\omega_\Omega}<\infty,$$ 
\end{itemize}
 there exists $u\in L^2_{n,q-1}(\Omega, E; \omega_\Omega, he^{-\psi})$ such that 
    \[
     \overline{\partial}  u=g \text{ and } 
    \int_\Omega |u|^2_{\omega_\Omega, h}e^{-\psi} \,dV_{\omega_\Omega}\leq \int_\Omega \langle B^{-1}_{\omega_\Omega,\psi,\theta}g, g\rangle_{\omega_\Omega,h}e^{-\psi} \,dV_{\omega_\Omega},
    \]
    where $B_{\omega_\Omega,\psi,\theta}=[(\deldel\psi +\theta )\otimes \Id_E,\Lambda_{\omega_\Omega}]$. 
\smallskip

(2) 
We define \textit{the local $\theta$-Nakano positivity in the sense of $L^2$-estimates}, denoted by 

    \[
    \ai\Theta_h \geq^{L^2}_{\loc\Nak} \theta \otimes \Id_{E} \text{ on } X,
    \]
as follows: For any point $x\in X$, there exists an open neighborhood $U$ of $x$ such that 
$
    \ai\Theta_h \geq^{L^2}_{\glo\Nak} \theta \otimes \Id_{E} \text{ on } U.
$
\end{defi}

\begin{defi}[Singular Nakano positivity for torsion-free sheaves]\label{def-dbarNak-sheaf}
Let $\mathcal{E}$ be a torsion-free sheaf and $\theta$ be a continuous $(1,1)$-form 
on a complex manifold $X$.
Let $X_{\mathcal{E}}$ denote the largest Zariski open set where $\mathcal{E}$ is locally free.
Consider a singular Hermitian metric $h$ on $\mathcal{E}$
(i.e.,\,a singular Hermitian metric on the vector bundle $\mathcal{E}|_{X_{\mathcal{E}}}$).
Assume that the function $|u|^2_{h^*}$ is upper semi-continuous for any local section $u$ of $(\mathcal{E}|_{X_{\mathcal{E}}})^*$.
% We say that the sheaf $(\mathcal{E}, h)$ is (global/local) $\theta$-Nakano positive in the sense of $L^2$-estimates
% if $(\mathcal{E}|_{X_{\mathcal{E}}}, h)$ satisfies this property.
We say that the sheaf $(\mathcal{E}, h)$ is global $\theta$-Nakano positive in the sense of $L^2$-estimates if $(\mathcal{E}|_{X_{\mathcal{E}}}, h)$ is global $\theta$-Nakano positive.
We also say that the sheaf $(\mathcal{E}, h)$ is local $\theta$-Nakano positive in the sense of $L^2$-estimates
if for any point $x\in X$, there exists an open neighborhood $U$ of $x$ such that $(\mathcal{E}|_{X_{\mathcal{E}}\cap U}, h)$ is global $\theta$-Nakano positive.
\end{defi}

The following lemmas play an important role in the proof of our main result.
For the reader's convenience, we provide the proofs and relevant references below.
Lemma \ref{lem-simplel2correcting} is directly derived from \cite[Lemma 9.10]{GMY23}
by setting $\eta=1$ and $g=k$ in {\cite[Lemma 9.10]{GMY23}}, 
and then taking the limit as $k \to \infty$.
Lemma \ref{lem-HermitianDaisyou} is derived from 
the proof of \cite[Chapter V\hspace{-1.2pt}I\hspace{-1.2pt}I\hspace{-1.2pt}I, Lemma (6.3)]{Dem-book}.

\begin{lemm}[{cf.\,\cite[Lemma 9.10]{GMY23}}]\label{lem-l2correcting}
    Let $X$ be a complete K\"ahler manifold with a $($not necessarily complete$)$ K\"ahler form $\omega$, 
    and $(Q,h)$ be a vector bundle with a smooth Hermitian metric $h$. 
Let   $\eta > 0$ and $g >0 $ be smooth functions on $X$ such that $(\eta+g)$ and $(\eta+g)^{-1}$ are bounded. 
Let $\lambda\geq 0$ be a bounded continuous function on $X$ such that $(B+\lambda I)$ is semi-positive definite everywhere on $\wedge^{n,q}T^* X\otimes Q$, 
where $n:=\dim X$ and $B:=B^{n,q}_{h,\eta,g,\omega}=[\eta\ai\Theta_{Q,h}-\deldel\eta-\ai g\partial\eta\wedge\overline{\partial} \eta, \Lambda_\omega]$. 

Then, for a given form $v\in L^2_{n,q}(X, Q; \omega, h)$ 
with 
$$\text{
$\overline{\partial}  v=0$ and $\int_X \langle (B+\lambda I)^{-1}v, v\rangle_{\omega, h}\,dV_\omega<\infty$, 
 }$$  
   there exist an approximate solution $u\in L^2_{n,q-1}(X, Q; \omega, h)$ and a correcting term 
   $\tau \in L^2_{n,q}(X, Q; \omega, h)$ such that 
$\overline{\partial}  u+P_h (\sqrt{\lambda}\tau )=v$ and   \[
\int_X (\eta+g^{-1})^{-1}|u|^2_{\omega,h} \,dV_\omega+ \int_X |\tau|^2_{\omega,h}\,dV_\omega \leq \int_X \langle (B+\lambda I)^{-1}v, v\rangle_{\omega,h} \,dV_\omega, 
    \]
where $P_h\colon  L^2_{n,q}(X, Q; \omega, h)\to \Ker \overline{\partial} $ is the orthogonal projection.
\end{lemm}

\begin{lemm}[{cf.\,\cite[Lemma 3.2]{ZZ18}}]\label{lem-simplel2correcting}
    Let $X$ be a complete K\"ahler manifold with a $($not necessarily complete$)$ K\"ahler form $\omega$, 
    and $(Q,h)$ be a vector bundle with a smooth Hermitian metric $h$. 
    Let $\delta\geq 0$ be a semi-positive constant such that $(B+\delta I)$  is semi-positive definite everywhere on $\wedge^{n,q}T^* X\otimes Q$, where $B:=B^{n,q}_{Q,h,\omega}=[\ai\Theta_{Q,h}, \Lambda_\omega]$. 

    Then, for a given form $v\in L^2_{n,q}(X, Q; \omega, h)$ 
with 
$$\text{
$\overline{\partial}  v=0$ and $\int_X \langle (B+\delta I)^{-1}v, v\rangle_{\omega, h}\,dV_\omega<\infty$, 
 }$$  
   there exist an approximate solution $u\in L^2_{n,q-1}(X, Q; \omega, h)$ and a correcting term 
   $\tau \in L^2_{n,q}(X, Q; \omega, h)$ such that 
$\overline{\partial}  u+\sqrt{\delta}\tau =v$ and 
\[
\int_X|u|^2_{\omega,h} \,dV_\omega+ \int_X |\tau|^2_{\omega,h}\,dV_\omega \leq \int_X \langle (B+\delta I)^{-1}v, v\rangle_{\omega,h} \,dV_\omega.
\]
\end{lemm}

\begin{comment}
Here, in general, the curvature operator \([\ai\Theta_{Q,h}, \Lambda_\omega]\) on \(\wedge^{p,q}T^*X\otimes Q\) is denoted by \(A^{p,q}_{Q,h,\omega}\) or simply $A_{h,\omega}$, and when \(A_{h,\omega}\) is positive (resp. semi-positive) definite, we simply write \(A_{h,\omega} > 0\) (resp. $\geq0$). 
Then, by \cite[Chapter VIII, Lemma (6.3)]{Dem-book}, we have the following lemma: 

\begin{lemm}[{\cite[Chapter VIII, Lemma (6.3)]{Dem-book}}]\label{lem-HermitianDaisyou}
    Let $(E,h)$ be a holomorphic Hermitian vector bundle over $X$ and $\omega, \gamma$ be Hermitian metrics on $X$ such that $\gamma\geq \omega$. 
    For every $u\in \wedge^{n,q}T^* X\otimes E$ with $q\geq 1$, we have that 
    $|u|^2_{\gamma,h} dV_\gamma \leq |u|^2_{\omega,h} dV_\omega$ and that if $A_{h,\omega}=A^{n,q}_{E,h,\omega}>0$ (resp. $\geq0$) then 
    \[
A_{h,\gamma}>0 \,\,(\text{resp.}\,\geq0), \hspace{5mm} \langle A_{h,\gamma}^{-1}u, u\rangle_{\gamma,h} dV_\gamma \leq  \langle A_{h,\omega}^{-1}u, u\rangle_{\omega,h} dV_\omega.
    \]
    %
\end{lemm}

Let $E$ be a holomorphic vector bundle over $X$.
Here, for any smooth operator $\theta_E\in \cal{C}^\infty(X,\wedge^2T^*X\otimes \mathrm{Hom}(E,E))$ with $\theta_E^*=-\theta_E$, the curvature operator \([\ai\theta_E, \Lambda_\omega]\) on \(\wedge^{p,q}T^*X\otimes E\) is denoted by \(B^{p,q}_{\theta_E,\omega}\) or simply $B_{\theta_E,\omega}$, and when \(B_{\theta_E,\omega}\) is positive (resp. semi-positive) definite, we simply write \(B_{\theta_E,\omega} > 0\) (resp. $\geq0$). 
\end{comment}

\begin{lemm}\label{lem-HermitianDaisyou}
    Let $X$ be a complex manifold of dimension $n$ 
    and $\omega_1,\omega_2$ be Hermitian forms on $X$ with $\omega_1\geq\omega_2$. 
    Let $(E, h)$ be a vector bundle on $X$ with a smooth Hermitian metric $h$ and $\theta_E\in C^\infty(X,\wedge^2T^*X\otimes \mathrm{Hom}(E,E))$ be a smooth operator with $\theta_E^*=-\theta_E$.
Assume that the operator $B_{\theta_E,\omega_2}:=[\ai\theta_{E},\Lambda_{\omega_2}]$ is positive 
$($resp. semi-positive$)$ definite on $\wedge^{n,q}T^*X\otimes E$, 
which we simply write as \(B_{\theta_E,\omega_2} > 0\) $($resp. $\geq0$$)$. 
Then, for any $(n,q)$-form $u\in\wedge^{n,q}T^*X\otimes E$ with $q\geq1$, 
we have 
    \[
    B_{\theta_E,\omega_1}>0 \,\,(\text{resp.}\,\geq0) \text{ and }  \langle B_{\theta_E,\omega_1}^{-1}u, u\rangle_{\omega_1,h} dV_{\omega_1} \leq  \langle B_{\theta_E,\omega_2}^{-1}u, u\rangle_{\omega_2,h} dV_{\omega_2}.
    \]
\end{lemm}

The following lemma is well-known and often applied in the case where $E$ is a line bundle. 
In this paper, we need a generalization of this lemma to vector bundles. While the proof is straightforward, 
we provide a proof for the general case to ensure completeness.

\begin{lemm} \label{lemm-L2}
Let $(E, h)$ be a singular Hermitian vector bundle on a complex manifold $X$. 
Assume that the norm $|t|_{h^{*}}$ with respect to $h^{*}$ is locally bounded 
for any  smooth section $t$ of $E^*$, 
where $h$ is the dual metric on the dual vector bundle $E^{*}$. 
Then, we have: 
\begin{itemize}
\item[$(1)$] Let $s=\sum_{i=1}^{r}s_{i} e_{i}$ be a $($not necessarily holomorphic or smooth$)$ 
section of $E$, where $\{e_{i}\}_{{i=1}}^{r}$ is a local frame of $E$. 
For any $p$ with $1 \leq p \leq \infty$,
the following holds: If $|s|_{h}$ is locally $L^{p}$-integrable,
then the absolute value $|s_i|$ of each component $s_{i}$ is also locally $L^{p}$-integrable.

\item[$(2)$] Let $\omega_X$ be a Hermitian form on $X$.
Let $u$, $v$ be $E$-valued-differential forms on $X$ 
such that $|u|_{h, \omega_X}$ is locally $L^{1}$-integrable and $|v|_{h,\omega_X}$ is locally $L^{2}$-integrable. 
If $\Dbar u=v$ holds on a $($non-empty$)$  Zariski open set $X_{0} \subset X$ in the sense of distributions, 
then $\Dbar u=v$  also holds on $X$. 
\end{itemize}
\end{lemm}
\begin{proof}
Note that the singular Hermitian metric $h$ is not defined on a set of Lebesgue measure zero, 
thus the lemma and its proof are valid almost everywhere on $X$.  
The norm $|s|_{h}$ is characterized as follows: 
$$
|s|_{h}=\sup\{ |\langle s, t \rangle_{\rm{pair}}| \, \mid\, t \text{ is a section of $E^{*}$ with }
|t|_{h^{*}} \leq 1\}, 
$$
where $\langle \bullet, \bullet \rangle_{\rm{pair}}$ is the point-wise natural pairing. 
By assumption, we can find a constant $\delta>0$ with $|\delta e^{*}_{i}| \leq 1$, 
where $\{e^{*}_{i}\}_{i=1}^{r}$ is the dual frame of $\{e_{i}\}_{{i=1}}^{r}$. 
Consequently, conclusion (1) follows from the inequality:  
$$ 
\delta |s_{i}| = |\langle s, \delta e^{*}_{i} \rangle_{\rm{pair}}| \leq |s|_{h}. 
$$
Conclusion (2) is then derived from Conclusion (1) and \cite[Lemma 6.9]{Dem82}.
\end{proof}

\section{Proof of the main results}\label{Sec-3}

In this section, we first establish Theorem \ref{thm-technical}, 
solving the $\Dbar$-equation with optimal $L^2$-estimate under Setting \ref{setup}. 
At the end of this section, as applications of Theorem \ref{thm-technical}, 
we prove Theorem \ref{thm-main} and Corollary \ref{cor-vanish}.

\begin{setup}\label{setup}
Under the setting of Theorem \ref{thm-main}, we further assume 
that $X$ is a K\"ahler manifold and the direct image sheaf 
$\mathcal{E}:=f_{*}(\mathcal{O}_{X}(K_{X/Y}+L) \otimes \I(h))$ is a locally free on $Y$. 
The locally free sheaf $\mathcal{E}$ is identified with a vector bundle on $Y$ and often written as $E$.
Furthermore, we assume that 
\begin{itemize}
\item $Y$ is a weakly pseudo-convex K\"ahler manifold;  
\item $\omega_Y$ is a K\"ahler form on $Y$; 
\item $\psi$ is a smooth function on $Y$ such that $\theta_\psi:= \theta + \deldel \psi >0$; 
\item $q$ is a positive integer with $1\leq q\leq m:=\dim Y$; 
\item $g\in L^2_{m,q}(Y, E; \omega_Y, He^{-\psi})$  is an $E$-valued $\Dbar $-closed form 
with 
$$
\int_Y \langle B^{-1}_{\omega_Y, \theta_\psi}\, g, g\rangle_{\omega_Y,H}e^{-\psi}\, dV_{\omega_Y}<\infty,
$$ 
where $B_{\omega_Y,\theta_\psi}=[\theta_\psi \otimes \Id_E,\Lambda_{\omega_Y}]$ and 
$H$ is the canonical $L^2$-metric on $E$ 
(see \cite{HPS18} or the proof of Theorem \ref{thm-technical} for the precise definition). 

\end{itemize}
Note that different from Definition \ref{def-dbarNak}, 
we do not assume that $Y$ is Stein nor that $Y$ admits  a trivialization $E \cong Y \times \C^r$. 
\end{setup}

Theorem \ref{thm-technical} is the main technical result of this paper and 
solves the $\Dbar$-equation with optimal $L^2$-estimate under Setting \ref{setup}, 
from which Theorem \ref{thm-main} and Corollary \ref{cor-vanish} can be deduced. 
The assumptions of Setting \ref{setup} appear somewhat technical,
but they are needed to prove Theorem \ref{thm-main} and Corollary \ref{cor-vanish} in a unified way.
The canonical $L^2$-metric $H$ on $E$ is known to be Griffiths $\theta$-semi-positive (see \cite{HPS18}), 
but  we never use this property in the proof of Theorem \ref{thm-technical}.

\begin{theo}\label{thm-technical}
Under Setting \ref{setup}, there exists $u\in L^2_{m,q-1}(Y, E; \omega_Y, He^{-\psi})$ such that 
$$ \Dbar  u=g \quad \text{ and } \quad
\int_Y |u|^2_{\omega_Y,H}e^{-\psi} \,dV_{\omega_Y}\leq \int_Y \langle B^{-1}_{\omega_Y,\theta_\psi}g, g\rangle_{\omega_Y,H}e^{-\psi} \,dV_{\omega_Y}. 
$$
\end{theo}

\begin{proof}
The proof can be divided into four steps.

\begin{step}[Strategy of the proof]\label{step1}
The purpose of this step is to explain the notation and the strategy of the proof. 

Let $(t_1,\ldots,t_m)$  be a  coordinate on an open subset $U \subset Y$. 
Then, the $E$-valued $(m,q)$-form $g$ can be locally written as 
$$
g=\sum_{|J|=q} g_J \,dt\wedge d \overline{t}_J 
\text{ and } g_J = \sum_{i=1}^{r} g_{J,i} e_{i}, 
$$
where \begin{itemize}
\item[$\bullet$] $J$ runs through the multi-indices of degree $q$;  
\item[$\bullet$] $dt:=\,dt_1\wedge\cdots\wedge \,dt_m$ and 
$d\overline{t}_J:=d\overline{t}_{j_{1}}\wedge\cdots\wedge d\overline{t}_{j_{q}}$ for $J=(j_{1}, j_{2}, \cdots, j_{q})$; 
\item[$\bullet$] $g_J$ is a section of $\mathcal{E}|_{U}$; 
\item[$\bullet$] $\{e_{i}\}_{i=1}^{r}$ is a frame of $\mathcal{E}|_{U}$ and $r:=\rank \mathcal{E}$; 
\item[$\bullet$] $g_{J,i}=g_{J,i}(t)$ is a function defined on $U\subset Y$.  
\end{itemize}
The section $g_J$ can be regarded as a 
section of $K_{X/Y}\otimes L \otimes \mathcal{I}(h)$ on the inverse image $f^{-1}(U) \subset X$ 
by the definition of direct image sheaves,  
and thus $g_J \otimes f^{*}dt$ can be identified with a section of $K_{X}\otimes L \otimes \mathcal{I}(h)$ 
since $f^{*}dt$ is the local frame of $f^{*}K_{Y}$. 
Under this identification, 
the $L$-valued $(n+m,q)$-form $\widetilde{g}$ is defined as follows: 
\begin{align}\label{eq-defg}
\widetilde{g}:=\sum_{|J|=q}  \big(g_J \otimes  f^{*}dt \big) \wedge f^{*} d\overline{t}_J 
\end{align}
where $n$ is the fiber dimension of $f \colon X \to Y$. 
Although $\widetilde{g}$ is defined by a local description of $g$, 
we can easily check that $\widetilde{g}$ is globally defined on $X$ 
as an $L$-valued $(n+m,q)$-form.

The strategy of the proof is as follows: 
In Step \ref{step2}, we show that $\widetilde{g}$ is $\Dbar$-closed 
and $L^2$-integrable in a suitable sense, by comparing $\widetilde{g}$ to $g$. 
In Step \ref{step3}, we solve the $\Dbar$-equation $\Dbar \widetilde{u} = \widetilde{g}$ on $X$, 
where after we apply Demailly's approximation theorem to $h$ and  
the result for a ``twisted'' $\Dbar$-equation, 
we take the limit to obtain the desired solution  $\widetilde{u}$ 
focusing the positivity in horizontal direction. 
In Step \ref{step4}, we finally descend $\widetilde{u}$ to an $E$-valued $(m,q-1)$-form  $u$  on $Y$ 
and confirm that $u$ satisfies the desired properties in Theorem \ref{thm-technical}, 
where the uniform $L^2$-estimate obtained in Step \ref{step3} plays a crucial role. 
\end{step}

\begin{step}[$L^{2}$-integrability and $\Dbar$-closedness of $\tilde{g}$]\label{step2}

For a fixed K\"ahler form $\omega_{X}$ on $X$, 
we  define  
the K\"ahler form $\omega_{\e}$ on $X$ by 
$$
\omega_{\e}:=f^{*}\omega_{Y}+\e \omega_{X}. 
$$ 
The purpose of this step is to prove Claim \ref{claim1}. 
To this end, we investigate the construction of  $\widetilde{g}$ and 
the canonical $L^2$-metric $H$ in detail. 

\begin{claim}\label{claim1}
Under the same situation as above, we have$:$
\begin{itemize}
\item[(1)] $\widetilde{g}$ is a  $\Dbar$-closed  and 
$L\otimes \mathcal{I}(h)$-valued  $(n+m, q)$ on $X$$;$ 

\item[(2)] For a positive integer $\e>0$, we have
\begin{align*}
 \int_{X} \langle B^{-1}_{\omega_\varepsilon, f^* \theta_\psi} \widetilde{g},\widetilde{g}\rangle_{\omega_\varepsilon, h}e^{-f^*\psi} \,dV_{\omega_\varepsilon} 
\leq  \int_Y \langle B^{-1}_{\omega_Y, \theta_\psi}g,g\rangle_{\omega_Y,H}e^{-\psi}\,dV_{\omega_Y} < \infty, 
\end{align*}
where $B_{\omega_\varepsilon, f^*\theta_\psi}:=[f^*\theta_\psi\otimes\Id_L,\Lambda_{\omega_\varepsilon}]$. 
\end{itemize}
\end{claim}

\begin{proof}
Following \cite[Section\,22]{HPS18}, we first recall the definition of canonical $L^2$-metrics. 
We take a Zariski open subset  $Y_0 \subset Y$ such that 
\begin{itemize}
\item[$\bullet$] $f \colon X_0:=f^{-1}(Y_0) \to Y_0$ is a smooth fibration over $Y_0$; 
\item[$\bullet$] The quotient sheaf $f_{*}(\mathcal{O}_{X}(K_{X/Y}+L))/\mathcal{E}$ is locally free on $Y_0$; 
\item[$\bullet$] $\mathcal{E} $ has the base change property. 
\end{itemize}
Then, for any $y \in Y_0$, we have the inclusions
$$
H^0(X_y, \mathcal{O}_{X_y} (K_{X_y}+L) \otimes \mathcal{I}(h|_{X_y}))
\subset 
E_y
\subset 
H^0(X_y, \mathcal{O}_{X_y}(K_{X_y}+L)) 
$$
under the natural identification $K_{X/Y} |_{X_y} \cong K_{X_y}$. 
Furthermore, the norm $|\alpha|_H$ of $\alpha \in E_y$ can be described by the fiber integral: 
$$
|\alpha|^2_H :=\int_{X_y} |\alpha|^2_{\omega_{X_y}, h} \, d V_{\omega_{X_y}}
$$
where $X_y$ is the fiber at $y \in Y_0$ and $\omega_{X_y}:={\omega_X}|_{X_y}$. 
By the property of $(n,0)$-forms,
the integrand of the right-hand side can be written as 
$$
|\alpha|^2_{\omega_{X_y}, h} \, d V_{\omega_{X_y}} = 
|\alpha'|^2_{h}  \,c_n dz \wedge d\overline{z}, 
$$  
where $\alpha=\alpha'\,dz$, $c_n:=\sqrt{-1}^{n^2}$ and $z$ is a local coordinate of the fiber $X_y$. 
In particular, the integrand does  not depend on the choice of K\"ahler forms on $X_y$. 

(1) It is obvious that $\widetilde{g}$ is 
an $L\otimes \mathcal{I}(h)$-valued $(n+m, q)$-form by construction (see \eqref{eq-defg}). 
We now show that $\Dbar \widetilde{g}=0$ on $X_{0}$. 
Since $f \colon X \to Y$ is a smooth fibration at $y_0 \in Y_0$, 
we can take a local coordinate $(z_{1}, \ldots, z_{n})$ 
of the smooth fiber $f^{-1}(y_0)$  
so that  
$$
(z_1,\cdots,z_n, f^*t_1,\cdots,f^*t_m)=(z_1,\cdots,z_n, t_1,\cdots,t_m)
$$ 
gives a local coordinate of $X$. 
Throughout the proof, we denote $f^*t_j$ by $t_j$ when considering it as part of the coordinates. 
However, we retain the notation $f^*t_j$ when emphasizing that it is the pullback of $t_j$.
The $L$-valued $(n+m,0)$-form $ g_J \otimes  f^{*}dt$ 
can be written as 
$$
g_J \otimes  f^{*}dt =\sum_{i=1}^r g_{J,i}  \,e_{i} \otimes  f^{*}dt 
= G_J\,dz \wedge \,dt  
$$
where $G_J=G_J(t,z)$ is a local section of $L$. 
Hence, we obtain 
$$
\widetilde{g}=\sum_{|J|=q}  G_{J} \,dz \wedge \,dt  \wedge d\overline{t}_{J}. 
$$
The function $g_{J,i}=g_{J,i}(t)$ does not depend on the fiber coordinate $z$ and 
$e_i$ is a holomorphic section, 
and thus we can deduce that $\Dbar_{z} G_J=0$ holds, which indicates that $\Dbar_z \widetilde{g}=0$. 
Furthermore, 
since we have 
 $$ 
\sum_{|J|=q}
\sum_{i=1}^r 
\sum_{k=1}^m
 \frac{\partial g_{J,i}}{\partial \overline{t}_k } e_i \otimes \,d \overline{t}_k \wedge d \overline{t}_J =0 
$$
by the assumption $\Dbar g=0$, 
we can see that 
\begin{align*}
\Dbar_{t} \widetilde{g} &= \sum_{|J|=q} \sum_{k=1}^m  \frac{\partial G_{J}(t,z)}{\partial \overline{t}_k} \,d \overline{t}_k \wedge dz \wedge \,dt  \wedge d\overline{t}_{J}\\
&= \sum_{|J|=q} \sum_{k=1}^m \sum_{i=1}^r  \frac{\partial g_{J,i}}{\partial \overline{t}_k } \,e_i 
\otimes  d \overline{t}_k \wedge dz \wedge dt  \wedge d \overline{t}_J =0. 
\end{align*}
Hence, we obtain $\Dbar \widetilde{g}=0$ on $X_0$. 
The property of $\Dbar \widetilde{g}=0$ can be extended to the ambient space $X$ by the (local) $L^2$-integrability of $\widetilde{g}$ 
(see \cite[Lemma 6.9]{Dem82}). 
This finishes the proof of (1).

(2) 
It is sufficient to prove the desired inequality for the integrals 
after we replace $f \colon X \to Y $ with $f \colon X_0 \to Y_0$. 
Let $\{\lambda_j\}_{j=1}^{m}$ be eigenvalues of $\theta_\psi$ with respect to $\omega_Y$. 
We take a coordinate $(t_1,\ldots,t_m,z_1,\ldots,z_n)$ around a fixed point $x_0\in X$ with $f(x_0)=y_0$ such that 
\begin{align*}
\theta_\psi=\sqrt{-1} \sum^m_{j=1}\lambda_j\,dt_j\wedge d\overline{t}_j  
\quad \text{ and } \quad 
\omega_{Y}=\sqrt{-1} \sum^m_{j=1}\,dt_j\wedge d\overline{t}_j  \quad \text{ at } y_0. 
\end{align*}
We first show that  the integrand $I_{\e}$ of the left-hand side in Claim \ref{claim1} (2),  
defined by 
$$
I_{\e}:=
 \langle B^{-1}_{\omega_\varepsilon, f^* \theta_\psi} \widetilde{g},\widetilde{g}\rangle_{\omega_\varepsilon, h}e^{-f^*\psi}\,dV_{\omega_\varepsilon}, 
$$
converges to $I$ monotonically from below, 
where 
$$
I:=\sum_{|J|=q} \Bigl(f^* \sum_{j\in J} \lambda_j \Bigr)^{-1}  
 \cdot  f^{*}|d\overline{t}_J|^2_{\omega_Y}  \cdot  
|G_J |_{h}^{2}  \cdot  e^{-f^*\psi} \, c_{n+m} dz \wedge \,dt \wedge 
d\overline{z} \wedge \,d\overline{t}. 
$$
Note that $I$ is a priori defined by local coordinates at each point $y_0$, 
but $I$ is independent of this choice of the coordinates and globally defined,
since $I$ is the limit of the globally defined $I_{\varepsilon}$.
 To this end, we consider the K\"ahler form $\omega_{\st}$ locally defined around  
$x_{0} \in X_{y_0}$: 
$$
\omega_{\st}:=\sum_{i=1}^{m} \sqrt{-1}\,dt_{i} \wedge d\overline{t}_{i} + 
\sum_{j=1}^{n}\sqrt{-1}\, dz_{j} \wedge d\overline{z}_{j}
$$ 
and take a positive constant $C>0$ so that 
$(1/C)\omega_{\st} \leq \omega_{X} \leq C \omega_{\st}$. 
Lemma \ref{lem-HermitianDaisyou} enables us to compare the  norms with respect to $\omega_\e$, 
$f^*\omega_Y+C\e \omega_{\st} $, and $f^*\omega_Y+(1/C)\e \omega_{\st} $. 
This implies that 
it is sufficient to show that $I_{\e,r}$ converges to $I$, 
where $r>0$, $\omega_{\e,r}:=f^*\omega_Y+\e r \omega_{\mathrm{st}}$ and 
\[
I_{\e,r}:=\langle B^{-1}_{\omega_{\varepsilon, r}, f^* \theta_\psi} \widetilde{g},\widetilde{g}\rangle_{\omega_{\e,r}, h}e^{-f^*\psi}\,dV_{\omega_{\e,r}}.
\]
A straightforward calculation yields 
$$
B^{-1}_{\omega_{\varepsilon, r}, f^* \theta_\psi} \widetilde{g}
=\sum_{|J|=q}\Bigl(\sum_{j\in J}  \frac{f^*\lambda_j }{\sqrt{1+\e r}} \Bigr)^{-1} G_J\, dz \wedge dt \wedge d\overline{t}_J  \quad \text{ at } y_0, 
$$
and thus, we can see that 
\begin{align*}
I_{\e,r}=\sum_{|J|=q} \Bigl(\sum_{j\in J}  \frac{f^*\lambda_j }{\sqrt{1+\e r}} \Bigr)^{-1} \cdot 
 |f^*d\overline{t}_{J}|^2_{\omega_{\e, r}} \cdot 
|G_J |_{h}^{2}e^{-f^*\psi} \, c_{n+m} dz \wedge \,dt \wedge d\overline{z} \wedge \,d\overline{t} \quad \text{ at } y_0. 
\end{align*}
Since $|f^*d\overline{t}_{J}|^2_{\omega_{\e, r}}$ and  $f^{*}\lambda_j/ \sqrt{1+\e r}$ 
converges to $f^{*}|d\overline{t}_{J}|^2_{\omega_Y}$ and $f^{*}\lambda_j$ respectively. 
Hence, Lemma \ref{lem-HermitianDaisyou} shows that  $I_{\e}$ converges to $I$ monotonically from below.

By the definition of the canonical $L^2$-metric $H$ on $\mathcal{E}$, 
we can see that 
\begin{align}\label{eq-defnorm}
| g_{J} \,dt \wedge d\overline{t}_{J}|^2_{\omega_Y,H}
= |\,d\overline{t}_{J}|^{2}_{\omega_{Y}} \cdot |\,dt| ^{2}_{\omega_{Y}} \cdot \int_{X_{y}}  |G_{J}|^{2}_{h}  
\, c_{n} dz \wedge d\overline{z}.
\end{align}
Then, due to the same calculation as above, we can see that 
$$
B_{\omega_Y,\theta_\psi}^{-1}\,g=\sum_{|J|=q}(\sum_{j\in J}\lambda_j)^{-1} g_J\,dt \wedge d\overline{t}_J  \quad \text{ at } y_0. 
$$
Together with \eqref{eq-defnorm}, we obtain that 
\begin{align*}
&\langle B_{\omega_Y,\theta_\psi}^{-1}\,g,g\rangle_{\omega_Y,H}e^{-\psi}\,dV_{\omega_Y}\\
=& \Big(
\sum_{|J|=q} 
(\sum_{j\in J}\lambda_j )^{-1} \cdot   |\,d\overline{t}_{J}|^{2}_{\omega_{Y}}  \cdot  e^{-\psi}  \cdot 
\int_{X_{y}}  |G_{J}|^{2}_{h}  \, c_{n} dz \wedge d\overline{z} \Big) \cdot c_{m} \,dt \wedge d\overline{t}
\\
=& \int_{X_{y}} \Bigl(
\sum_{|J|=q} (f^* \sum_{j\in J} \lambda_j )^{-1} \cdot   f^{*}|d\overline{t}_{J}|^2_{\omega_Y} \cdot  
|G_J |_{h}^{2} e^{-f^*\psi} 
 \, c_{n} dz \wedge d\overline{z} \Bigr)
\cdot c_{m} dt \wedge d\overline{t}  \quad \text{ at } y_0.
\end{align*}
Hence, Fubini's theorem shows that 
\begin{align}\label{eq-1}
&\int_Y\langle B^{-1}_{\omega_Y,\theta_\psi}g,g\rangle_{\omega_Y,H}e^{-\psi}\,dV_{\omega_Y}
 \\
=&\int_{X}\Bigl(
\sum_{|J|=q} (f^* \sum_{j\in J} \lambda_j )^{-1} \cdot   f^{*}|d\overline{t}_{J}|^2_{\omega_Y} \cdot  
|G_J |_{h}^{2} e^{-f^*\psi} 
 \, c_{n} dz \wedge d\overline{z} \Bigr)\cdot c_{m} dt \wedge d\overline{t}.
 \notag
\end{align}
The desired conclusion follows from Lebesgue's monotone convergence theorem 
since $I_{\e}$ converges to the integrand in the right-hand side monotonically from below.
\end{proof}

\end{step}

\begin{step}[$\Dbar$-equations on $X$ with optimal $L^2$-estimate]\label{step3}

The purpose of this step is to prove Claim \ref{claim2}. 
To this end, we apply Demailly's approximation theorem to $h$, 
and then take the limit of appropriate  solutions of ``twisted'' $\Dbar$-equations.

\begin{claim}\label{claim2}
Under the same situation as above, there exists $\widetilde{u} \in L^2_{n+m,q-1}(X, L; \omega_X, h e^{-f^*\psi})$ 
such that $\Dbar\widetilde{u}=\widetilde{g}$ on $X$ and 
\begin{align*}
\int_{X}|\widetilde{u}|^2_{\omega_\varepsilon,h}e^{-f^*\psi}\,dV_{\omega_{\varepsilon}}
\leq\int_Y \langle B^{-1}_{\omega_Y,\theta_\psi}g,g\rangle_{\omega_Y,H}e^{-\psi}\,dV_{\omega_Y}
\end{align*}
for any $\e>0$. 
\end{claim}
\begin{proof}
Let \(\{Y_j\}_{j=1}^{\infty}\) be an open cover of \(Y\) such that 
$Y_{j}\Subset Y_{j+1}$ and \(Y_j \Subset Y \) hold for any $j$. 
We may assume that $Y_j$ is a sub-level set of an exhaustive psh function on $Y$ 
since $Y$ is weakly pseudo-convex. 
The inverse image \(X_j = f^{-1}(Y_j)\) also satisfies that $X_{j}\Subset X_{j+1}$ and \({X}_j \Subset X \).
By applying Demailly's approximation theorem \cite{Dem92} to $h|_{X_{j}}$ and $\omega_\e|_{X_j}$, 
we can take a family $\{h_{j, \e, \delta}\}_{\delta>0}$ of singular Hermitian metrics on $L|_{X_{j}}$
such that 
\begin{itemize}
\item [$(1)$] $h_{j, \e, \delta}$ has analytic singularities along a Zariski closed subset $Z_{j,\e, \delta}  \subset X_j$;  
\item [$(2)$] $\{h_{j, \e, \delta}\}_{\delta>0}$ increasingly converges to $h|_{X_{j}}$ as $\delta \searrow 0$; 
\item [$(3)$] $\sqrt{-1}\Theta_{h_{j, \e, \delta}}(L) \geq f^*\theta - \delta \omega_\e $ holds on $X_{j}$. 
\end{itemize}
Note that Demailly's approximation theorem remains valid for the relatively compact subset $X_{j} \Subset X$ 
(see \cite[Theorem 2.9]{Mat22} for the detailed argument). 
We consider 
the curvature operator $B_{j, \e, \delta}$ defined by 
$$
B_{j,\e,\delta}:=[\sqrt{-1} \Theta_{h_{j, \e, \delta}}(L) + f^{*} \idd \psi,\Lambda_{\omega_\varepsilon}] 
\text{ on } X_{j} \setminus Z_{j,\e, \delta}. 
$$ 
Then, we can easily check that  
\begin{align*}
B_{j, \e, \delta}+2\delta [\omega_\varepsilon, \Lambda_{\omega_\varepsilon}] 
>&B_{j, \e, \delta}+\delta[\omega_\varepsilon, \Lambda_{\omega_\varepsilon}] \\
\geq&[f^*\theta_\psi \otimes\Id_L,\Lambda_{\omega_\varepsilon}]
=B_{\omega_\varepsilon, f^* \theta_\psi }\geq0. 
\end{align*}
Note that $[\omega_\varepsilon, \Lambda_{\omega_\varepsilon}]=q\cdot\Id_L$ holds 
on $\wedge^{n+m,q}T^*X\otimes L$.
Thus, together with Claim \ref{claim1}, we can  deduce that 
\begin{align*}\label{eq-ineq}
& \int_{X_j\setminus Z_{j,\e, \delta}}\langle(B_{j, \e, \delta}+2q\delta\cdot\Id_L)^{-1}\widetilde{g},\widetilde{g}\rangle_{\omega_\varepsilon, h_{j, \e, \delta}}e^{-f^*\psi}\,dV_{\omega_\varepsilon} \notag \\
 \leq &  
 \int_{X} \langle B^{-1}_{\omega_\varepsilon, f^* \theta_\psi } \widetilde{g},\widetilde{g}\rangle_{\omega_\varepsilon, h}e^{-f^*\psi}\,dV_{\omega_\varepsilon} \notag  \\
\leq &  \int_Y\langle B^{-1}_{\omega_Y, \theta_\psi}g,g\rangle_{\omega_Y,H}e^{-\psi}\,dV_{\omega_Y}.
\end{align*}
Note that  $X_{j} \setminus Z_{j,\e, \delta}$ is a complete K\"ahler manifold 
by \cite[Theorem\,1.5]{Dem82} since $X_j$ is a sub-level set of an exhaustive psh function. 
By applying Lemma \ref{lem-simplel2correcting}, %
we can find an approximate solution $\widetilde{u}_{j, \e, \delta}$ and a correcting term $\tau_{j, \e, \delta}$: 
\begin{align*}
&\widetilde{u}_{j, \e, \delta}\in L^2_{n+n,q-1}(X_j\setminus Z_{j,\e, \delta},L;\omega_\varepsilon,h_{j, \e, \delta} e^{-f^*\psi}) \text{ and } \\
&\tau_{j, \e, \delta}\in L^2_{n+m,q-1}(X_j\setminus Z_{j,\e, \delta},L;\omega_\varepsilon,h_{j, \e, \delta} e^{-f^*\psi})
\end{align*}
satisfying that
\begin{align*}
\Dbar \widetilde{u}_{j, \e, \delta}+\sqrt{2q\delta }\,\tau_{j, \e, \delta}&=\widetilde{g} \,\,\text{on}\,\,X_j\setminus Z_{j,\e, \delta} \quad\text{and}\\
\int_{X_j\setminus Z_{j,\e, \delta}}|\widetilde{u}_{j,\e, \delta}|^2_{\omega_\varepsilon,h_{j, \e, \delta}}e^{-f^*\psi}\,dV_{\omega_\varepsilon}&+\int_{X_j\setminus Z_{j,\e, \delta}}|\tau_{j, \e, \delta}|^2_{\omega_\varepsilon,h_{j, \e, \delta}}e^{-f^*\psi}\,dV_{\omega_\varepsilon}\\
&\leq\int_Y\langle B^{-1}_{\omega_Y,\theta_\psi}g,g\rangle_{\omega_Y,H}e^{-\psi}\,dV_{\omega_Y}.
\end{align*}
Note that the right-hand side does not depend on the data $j, \e, \delta$. 
The first property of $\Dbar $-equation on $X_j\setminus Z_{j,\e, \delta}$ can be extended to $X_j$ 
by the (local) $L^2$-integrability (see \cite[Lemma 6.9]{Dem82}). 
Hereafter, we regard both $\widetilde{u}_{j, \e, \delta}$ and $\tau_{j, \e, \delta}$ 
as $L^{2}$-sections on the ambient space $X$ via the zero extensions. 
We will take an appropriate weak limit to construct the desired solution $\widetilde{u}$.

We first fix $j$, $\e$, and consider a weak limit as $\delta \to 0$. 
The $L^{2}$-norms $\|\widetilde{u}_{j, \e, \delta} \|_{\omega_\varepsilon,h_{j, \e, \delta}}$ 
and $\|\tau_{j, \e, \delta}\|_{\omega_\varepsilon, h_{j, \e, \delta}}$ are uniformly bounded  in $\delta$ 
and $\{h_{j, \e, \delta}\}_{\delta>0}$ monotonically increases as $\delta \searrow 0$. 
Thus, by the same argument as in \cite[Proof of Theorem 1.2]{IM24} (see also \cite{GMY23, Mat18, Mat22}), 
we can take a subsequence of $\{\widetilde{u}_{j, \e, \delta}\}_{\delta>0}$ 
(for which we use the same notation for simplicity) and 
$$
\widetilde{u}_{j, \e}\in L^2_{n+m,q-1}(X_j, L;\omega_\varepsilon, h e^{-f^*\psi})
$$ 
such that 
$\widetilde{u}_{j, \e, \delta}$ weakly converges to $\widetilde{u}_{j, \e} $ 
on $L^2_{n,q-1}(X, L; \omega_\varepsilon, h_{\delta_{0}} e^{-f^*\psi}) $ 
for a positive number $\delta_{0}>0$.  
Note that this subsequence and its weak limit is independent of $\delta_0>0$ by the diagonal argument. 
Furthermore,  by the lower semi-continuity of  the weak limit and property (2) of $\{h_{j, \e, \delta}\}_{\delta>0}$, 
we can see that for $\delta_0 >0$
\begin{align*}
\int_{X_j}|\widetilde{u}_{j, \e}|^2_{\omega_\varepsilon,h_{j, \e, \delta_0 }}
e^{-f^*\psi}\,dV_{\omega_\varepsilon}
&\leq \liminf_{\delta \to 0}
\int_{X_j}|\widetilde{u}_{j, \e, \delta}|^2_{\omega_\varepsilon,h_{j, \e, \delta_0 }}
e^{-f^*\psi}
\,dV_{\omega_\varepsilon}\\
&\leq \liminf_{\delta \to 0}
\int_{X_j}|\widetilde{u}_{j, \e, \delta}|^2_{\omega_\varepsilon,h_{j, \e, \delta}}e^{-f^*\psi}
\,dV_{\omega_\varepsilon}\\
&\leq\int_Y\langle B^{-1}_{\omega_Y,\theta_\psi}g,g\rangle_{\omega_Y,H}e^{-\psi}\,dV_{\omega_Y}.
\end{align*}
By applying  Fatou's lemma, we obtain that 
\begin{align*}
\int_{X_j}|\widetilde{u}_{j, \e}|^2_{\omega_\varepsilon,h} e^{-f^*\psi}\,dV_{\omega_\varepsilon}
&\leq \liminf_{\delta_{0} \to 0}
\int_{X_j}|\widetilde{u}_{j, \e}|^2_{\omega_\varepsilon,h_{j, \e, \delta_{0}}}e^{-f^*\psi}\,dV_{\omega_\varepsilon}\\
&\leq\int_Y\langle B^{-1}_{\omega_Y,\theta_\psi}g, g\rangle_{\omega_Y,H}e^{-\psi}\,dV_{\omega_Y}.
\end{align*}

On the other hand, in  the same way applied to $\{\tau_{j, \e, \delta}\}_\delta$, 
we may assume that $\{\tau_{j, \e, \delta}\}_{\delta>0}$ weakly converges to some $\tau_{j, \e}$. 
Then
it follows that 
$\sqrt{2q\delta}\,\tau_{j, \delta, \e}$ weakly converges to zero 
by $\delta  \to 0$. 
Consequently, by taking the weak limit as $\delta \to 0$, we obtain that 
\begin{align*}
\Dbar\widetilde{u}_{j,\e}&=\widetilde{g}\,\,\text{on}\,\,X_j\quad \text{and}\\
\int_{X_j}|\widetilde{u}_{j,\e}|^2_{\omega_\varepsilon,h}e^{-f^*\psi}\,dV_{\omega_\varepsilon}&\leq\int_Y\langle B^{-1}_{\omega_Y,\theta_\psi}g,g\rangle_{\omega_Y,H}e^{-\psi}\,dV_{\omega_Y}.
\end{align*}

We now fix $\e$, and consider a weak limit as $j \to \infty$. 
The right-hand side does not depend on $j$ and 
the norm $\int_{X_j}|\bullet|^2_{\omega_\varepsilon,h}e^{-f^*\psi}\,dV_{\omega_\varepsilon}$
monotonically increases as $j \nearrow \infty$. 
Therefore, for a fixed positive number $\e>0$, the same argument as above works for $\{\widetilde{u}_{j, \e}\}_{j=1}^\infty$. 
Hence, the subsequence $\{\widetilde{u}_{j, \e}\}_{j=1}^\infty$  
(for which we use the same notation)
weakly converges to some $\widetilde{u}_{\e} \in L^2_{n,q-1}(X,L;\omega_\varepsilon,h e^{-f^*\psi})$ 
and the weak limit $\widetilde{u}_{\e}$  satisfies that 
\begin{align*}
\Dbar\widetilde{u}_{\e}&=\widetilde{g}\,\,\text{on}\,\,X_j\quad \text{and}\\
\int_{X}|\widetilde{u}_{\e}|^2_{\omega_\varepsilon,h}e^{-\psi}
\,dV_{\omega_\varepsilon}&\leq\int_Y \langle B^{-1}_{\omega_Y,\theta_\psi}g,g\rangle_{\omega_Y,H}e^{-\psi}\,dV_{\omega_Y}.
\end{align*}
By noting that that $|w|^2_{\omega_\e} \,dV_{\omega_\e}$ 
monotonically increases as $\e \searrow 0$ for any $(n+m, q)$-form $w$, 
we can apply the same argument as above again, which yields the desired solution $\widetilde{u}$. 
\end{proof}
\end{step}

\begin{step}[Descent of $\widetilde{u}$ to $Y$]\label{step4}
The purpose of this step is to show that $\widetilde{u}$ determines an $E$-valued $(m,q-1)$ form $u$ on $Y$ 
with the desired properties in Theorem \ref{thm-technical} to complete the proof. 
To this end,  from the uniform $L^{2}$-estimate in Claim \ref{claim2}, 
we deduce that $\widetilde{u}$ can be locally written
\begin{align*}
\widetilde{u}=\sum_{|K|=q-1}U_K\, dz\wedge \,dt\wedge d\overline{t}_K 
\end{align*}
on  $X_{0}=f^{-1}(Y_{0})$, where $U_K$ is a local section of $E$, 
in other words, the term containing \( d\overline{z}_i \) never appear in the local expression. 
To derive a contradiction, we assume that 
$\widetilde{u}$ contains the term 
$U dz\wedge \,dt\wedge d\overline{z}_I \wedge d\overline{t}_L $, 
where $I \not = \emptyset$ and $L$ are multi-indices. 
We take the K\"ahler form $\omega_{\st}$ locally defined on $X_0$ as in Step \ref{step2},  
and the K\"ahler form $\omega_{\st, Y}:=\sqrt{-1}\sum_{i=1}^{m} \,dt_{i}\wedge d\overline{t}_{i} $ 
locally defined on $Y_0$. 
Fix a positive constant $C>0$ so that 
$\omega_{\e} \leq \omega'_{\e}:=C (f^{*}\omega_{\st, Y}+ \e \omega_{\st})$ 
holds. 
Then, by noting that $dt_{j}$ and $dz_{k}$ are orthogonal  with respect to the new metric $\omega'_{\e}$, 
we can see that 
\begin{align*}
|\widetilde{u}|^2_{h, \omega_\varepsilon}\,dV_{\omega_\varepsilon}
&\geq|\widetilde{u}|^2_{h,\omega'_\varepsilon}\,dV_{\omega'_\varepsilon}\\
&\geq|U \, dz\wedge  dt \wedge d\overline{z}_I \wedge d\overline{t}_L|^2_{h, \omega'_\varepsilon}\,dV_{\omega'_\varepsilon}\\
&=|U|_{h}^2 
|d\overline{z}_I|^2_{\omega'_\varepsilon} |d\overline{t}_L|^2_{\omega'_\varepsilon} 
 |dz\wedge \,dt|^2_{\omega_X}\,dV_{\omega_X}
\end{align*}
on an open set in $X_0$
Furthermore, by the definition of $\omega'_{\e}$, 
we can see that $|d\overline{z}_I|^2_{\omega'_\varepsilon} = 1/(C\e)^{|I|}$ 
and $|d\overline{t}_L|^2_{\omega'_\varepsilon} =1/(C(1+\e))^{|L|}$. 
This indicate that the integral of $|\widetilde{u}|^2_{h, \omega_\varepsilon}\,dV_{\omega_\varepsilon}$ 
is not uniformly bounded in $\e$, which contradicts Claim \ref{claim2}.

By $\Dbar \widetilde{u} = \widetilde{g}$ and the above local expression, 
we obtain that 
$$
\sum_{|K|=q-1} \big( \sum_{j=1}^{m}  \frac{\partial U_{K}}{\partial \overline t_{j}} \, d \overline{t}_{j} +
\sum_{k=1}^{n}  \frac{\partial U_{K}}{\partial \overline z_{k}}\, d \overline{z}_{k}
 \big) \wedge d\overline{t}_{K} = 
\sum_{|J|=q}  G_{J}   \,  d\overline{t}_{J}. 
$$
By considering the term containing $d \overline{z}_{k}$, 
we can see  that $U_K$ is holomorphic in the fiber coordinate $z_i$, 
which implies that $u_K(t):=U_K(t,\bullet)\,dz\in H^0(X_t,K_{X_t}\otimes L|_{X_t})$
Thus, we can conclude that 
$\widetilde{u}$  can be determined by the $E$-valued $(m,q-1)$-form $u$ on $Y_{0}$ defined by 
$$u:=\sum_{|K|=q-1}u_K(t)\,dt\wedge d\overline{t}_K.$$
Furthermore, by considering the term containing $d \overline{t}_{k}$, 
we can see that $\Dbar u= g$ holds on $Y_{0}$. 

We now consider the $L^2$-norm of $u$. 
Note that the norm $|d\overline{t}_{K}|^{2}_{\omega_{\e}} $ 
converges to $f^{*}|d\overline{t}_{K}|^{2}_{\omega_{Y}}$ by the same calculation as in Step \ref{step2} and 
$$
|\widetilde{u}|^2_{\omega_\e,h}\,dV_{\omega_\e}\nearrow \sum_{|K|=q-1} f^{*}|d\overline{t}_{K}|^{2}_{\omega_{Y}} \cdot |U_K|^2_h\, c_{n+m}dt\wedge dz\wedge d\bar{t} \wedge d\bar{z}
$$
as $\e\searrow 0$. 
Hence, by Fatou's lemma and the argument  in Step \ref{step2} based on Fubini's theorem, 
we can see that 
\begin{align*}
&\liminf_{\e \to 0}
\int_{X_{0}}|\widetilde{u}|^2_{\omega_{\e},h} \cdot  e^{-f^*\psi}\,dV_{\omega_{\e}} \\
\geq &
\int_{X_{0}} \sum_{|K|=q-1} 
f^{*}|d\overline{t}_{K}|^{2}_{\omega_{Y}}\cdot  |U_{K}|^{2}_{h} \cdot  e^{-f^*\psi}
\, c_{n+m} dz\wedge dt \wedge d\overline{z}\wedge \,d\overline{t}\\
=&
\int_{y\in Y_{0}} \sum_{|K|=q-1} |d\overline{t}_{K}|^{2}_{\omega_{Y}} e^{-\psi}
 \Big( \int_{X_{y}}  |U_{K}|^{2}_{h}\, c_{n} dz \wedge d\overline{z}  \Big)
 \, c_{m}  dt \wedge d\overline{t} \\
 =& \int_{Y_{0}} |u|^{2}_{\omega_{Y}, H} e^{-\psi} \,dV_{\omega_Y} 
\end{align*}
Hence, together with Claim \ref{claim2}, 
we obtain the desired $L^{2}$-estimate: 
\begin{align*}
\int_Y|u|^2_{\omega_Y,H}e^{-\psi}\,dV_{\omega_Y}\leq\int_Y\langle B^{-1}_{\omega_Y,\theta_\psi}g,g\rangle_{\omega_Y,H} e^{-\psi}\,dV_{\omega_Y}.
\end{align*}

\cite[Proposition 23.3]{HPS18} shows that 
the canonical $L^2$-metric $H$ satisfies the assumption in Lemma \ref{lemm-L2}, 
(i.e.,\,the norm $|t|_{H^{*}}$ is bounded for any section of $E^{*}$). 
This  can be proved 
without using the Griffiths $\theta$-semi-positivity (see the proof of \cite[Proposition 23.3]{HPS18}). 
Thus, even in this step,  
our proof never uses the Griffiths $\theta$-semi-positivity of $H$ (see Remark \ref{rem-Gposi}). 
By Lemma \ref{lemm-L2}, 
the $\Dbar$-equation $\Dbar u=g$ on $Y_{0}$ can be extended on $Y$, 
completing the proof.  
\end{step}
\vspace{-8mm}
\end{proof}

\begin{rem}\label{rem-Gposi}
As noted, the Griffiths $\theta$-semi-positivity of $H$ is not used in the proof of Theorem \ref{thm-main}.
More precisely, the only result from \cite{HPS18, PT18} that we use is \cite[Proposition 23.3]{HPS18} in the final step.
This result follows from the Ohsawa-Takegoshi $L^{2}$-extension theorem with a ``non-optimal'' $L^{2}$-estimate.
Due to \cite{BL16}, the Griffiths semi-positive of $H$ is equivalent to 
the Ohsawa-Takegoshi $L^{2}$-extension theorem with ``optimal'' $L^{2}$-estimate. 
In this sense, our proof does not depend on the Griffiths $\theta$-semi-positive essentially. 
\end{rem}

At the end of this section, as an application of Theorem \ref{thm-technical}
we give proofs for Theorem \ref{thm-main} and Corollary \ref{cor-vanish}. 

\begin{proof}[Proof of Theorem \ref{thm-main}]

We first remark that 
the canonical $L^2$-metric $H$ is lower semi-continuous by \cite{HPS18}. 
We emphasize that in the proof of Theorem \ref{thm-main}, 
the Griffiths semi-positivity is used only to check this lower semi-continuous. 
More precisely, we do not use the Griffiths semi-positivity
in checking that $H$ is  lower semi-continuous on $Y \setminus Z$, 
where $Z$ is a suitable Zariski closed subset (see \cite[Proposition 22.5]{HPS18}), 
but we use the Griffiths semi-positivity in extending this property to $Y$.

Let $Y_{\mathcal{E}} \subset Y$ be the maximal Zariski open subset where 
the torsion-free sheaf
$\mathcal{E}:=f_{*}(\mathcal{O}_{X}(K_{X/Y}+L) \otimes \I(h))$ is locally free. 
Take a Stein coordinate $\Omega \subset Y_{\mathcal{E}}$ 
equipped with the data as in Definition \ref{def-dbarNak}. 
If the inverse image $f^{-1}(\Omega)$ is K\"ahler, 
then the $\Dbar$-equation with optimal $L^{2}$-estimate  in Definition \ref{def-dbarNak}
can be solved by Theorem \ref{thm-technical}.
This K\"ahler condition is satisfied if $\Omega$ is sufficiently small, which proves conclusion (1). 
Also, this K\"ahler  condition is satisfied if $X$ is K\"ahler, 
which proves conclusion (2). 
\end{proof}

\begin{proof}[Proof of Corollary \ref{cor-vanish}]
Our strategy of the proof is based on the argument in \cite{Ohs84, Mat16}. 

We first fix a Stein open cover $\mathcal{U}:=\{U_{i}\}_{i\in I}$ of $Y$ 
and a partition $\{\rho_i \}_{i \in I}$ of unity associated with $\mathcal{U}$. 
Then, we consider the  diagram 
$$
\xymatrix{
H^{q}(X, K_{X}\otimes L \otimes \I(h))
\ar[r]^r \ar@{}[dr]|\circlearrowleft & H^{q}(X_{\mathcal{E}}, K_{X}\otimes L \otimes \I(h)) \\
\check{H}^{q}(\mathcal{U}, K_{Y}\otimes \mathcal{E}) \ar[u]_{\varphi}  \ar[r]^{\varphi_Y} & 
H^{q}(Y_{\mathcal{E}}, K_{Y}\otimes \mathcal{E})   \ar[u]_{\widetilde{\bullet}}, \\
}
$$
where $\mathcal{E}:=f_{*}(\mathcal{O}_{X}(K_{X/Y}+L) \otimes \I(h))$ 
and the above morphisms are defined as follows: 
The morphism $\widetilde{\bullet}$ is defined by 
the same operation $[\widetilde{g}]$ as in Step \ref{step1} of the proof of  Theorem \ref{thm-technical} 
by regarding $H^{q}(Y_{\mathcal{E}}, K_{Y}\otimes \mathcal{E}) $ as the $\Dbar$-cohomology group 
and  $r$ is the restriction map from $X$ to $X_{\mathcal{E}}:=f^{-1}(Y_{\mathcal{E}})$, 
where $Y_{\mathcal{E}} \subset Y$ is the maximal Zariski open where $\mathcal{E}$ is locally free. 
To explain the definitions of $\varphi$ and $\varphi_Y$, 
we take a $\rm{\check{C}}$ech cohomology class $[\{ \alpha_{i_{0}\cdots i_{q}} \}] 
\in \check{H}^{q}(\mathcal{U}, K_{Y}\otimes \mathcal{E}) $. 
Then, the $\Dbar$-cohomology class $\varphi_Y(\alpha)$ is defined by 
the $\Dbar$-closed $E$-valued  $(m, q)$-form $\varphi_Y(\alpha)$ 
constructed by the standard method: 
 $$
\varphi_Y(\alpha):= 
\Dbar \sum_{i_{q-1}} \rho_{i_{q-1}} \Big(
\cdots 
\Dbar \sum_{i_{1}} \rho_{i_{1}} \big(
\Dbar \sum_{i_{0}} \rho_{i_{0}} (\alpha_{i_{0}\cdots i_{q}}|_{U_{i_{0}\cdots i_{q}} \cap Y_{\mathcal{E}}})
\big)\Big), 
$$
where $\U{q}:=U_{i_{0}}\cap\cdots\cap U_{i_{q}}$ is a Stein open set. 
By the projection formula 
$$
f_*\mathcal{O}_X (K_{X}\otimes L \otimes \I(h)) \cong  K_{Y}\otimes \mathcal{E}, 
$$
we have the natural morphism
$$
H^{0}(\U{p}, f_{*}(K_{X}\otimes L \otimes \I(h) )) \xrightarrow{\quad \cong \quad}
H^{0}(f^{-1}(\U{p}), K_{X}\otimes L \otimes \I(h) ), 
$$
which corresponds to $\widetilde{\bullet}$ of the case $q=0$. 
Hence, we can regard $\alpha_{i_{0}\cdots i_{q}}$ as the section 
$$
\widetilde{\alpha}_{i_{0}\cdots i_{q}} \in H^{0}(f^{-1}(\U{p}), K_{X}\otimes L \otimes \I(h) ).$$ 
Then $\varphi(\alpha)$ is defined by the same way as  $\varphi_Y(\alpha)$
using  $\{ \widetilde{\alpha}_{i_{0}\cdots i_{q}} \}$ and $\{\widetilde{\rho}_{i} := f^{*} \rho_{i} \}_{i \in I}$. 
We can easily see that the above diagram is commutative. 
Furthermore,  the morphism $\varphi$ is injective by the same argument as in 
\cite[Subsection 3.3]{Mat16} (and \cite{Ohs84}). 

We will apply the proof of Theorem 3.2 to the case where 
$g:=\varphi_Y(\alpha)$, $\theta:=\omega_{Y}$, and $\psi=0$. 
Then, by construction and noting that $\theta_0=\omega_Y$, we can see that 
$$
\langle B^{-1}_{\omega_Y, \theta_0}\, g, g\rangle_{\omega_Y,H}
=\frac{1}{q} |g|^2_{\omega_Y,H}
$$
is locally $L^{2}$-integrable on $Y$ (not only on $Y_{\mathcal{E}}$)

Since $Y$ is a weakly pseudo-convex manifold, 
we can find an exhaustive psh $\psi$ function on $Y$ 
such that 
$\langle B^{-1}_{\omega_Y, \theta_\psi}\, g, g\rangle_{\omega_Y,H} e^{-\psi}$
is $L^{2}$-integrable on $Y$. 
Indeed, this can be achieved 
by replacing $\psi$ with $\chi \circ \psi $ and  noting that 
$ B^{-1}_{\omega_Y, \theta_\psi} \leq  B^{-1}_{\omega_Y, \theta_0}$, 
where $\chi \colon \mathbb{R} \to \mathbb{R}$ is  a rapidly increasing convex function.  
Thus, Claim \ref{claim2} shows that there exists $\widetilde{u}$ on $X_{\mathcal{E}}$ satisfying
\begin{align*}
\Dbar\widetilde{u}&=\widetilde{g}\,\,\text{on}\,\,X_{\mathcal{E}}\quad \text{and}\\
\int_{X_{\mathcal{E}}}|\widetilde{u}|^2_{\omega_\varepsilon,h}e^{-f^*\psi}\,dV_{\omega_{\varepsilon}}
&\leq\int_{Y_{\mathcal{E}}} \langle B^{-1}_{\omega_Y,\theta_\psi}g,g\rangle_{\omega_Y,H}e^{-\psi}\,dV_{\omega_Y}.
\end{align*}
The smooth function $f^{*}\psi$ is defined on $X$ (not only on $X_{\mathcal{E}}$).
Thus, the $\Dbar$-equation $\Dbar\widetilde{u}=\widetilde{g}$ can be extended to $X$, 
which implies that $\varphi(\alpha)$ is $\Dbar$-exact. 
This finishes the proof since $\varphi$ is injective and 
$\widetilde{g}$ is a representation of the $\Dbar$-cohomology $\varphi(g)$. 
 \end{proof}
 
\section{On direct image sheaves of relative pluri-canonical bundles}\label{Sec-4}

In this section, 
we consider the relative pluricanonical bundle $mK_{X/Y}+L$ 
under the situation of Theorem \ref{thm-main}, 
and establish the singular Nakano positivity of the Narasimhan-Simha metric 
on its direct image sheaf (see Corollary \ref{cor-main}).
The proof is an application of Theorem \ref{thm-main}
using the Bergman kernel metric and Narasimhan-Simha metrics developed in \cite{PT18}.

\begin{cor}\label{cor-main}
Consider the situation as in Theorem \ref{thm-main}. 
For a fixed positive integer $m \in \mathbb{Z}_+$, we consider the $m$-th Narasimhan-Simha metric $H_m$ on the direct image sheaf
$$
\mathcal{E} := f_{*}(\mathcal{O}_{X}(mK_{X/Y} + L) \otimes \mathcal{I}( B_{m}^{-\frac{m-1}{m}} \cdot h^{\frac{1}{m}})),
$$
where $B_{m}^{-1}$ denotes the $m$-th Bergman kernel metric on $mK_{X/Y}+L$. 
$($See \cite[3.1]{PT18} for the definitions of Bergman kernel metrics and Narasimhan-Simha metrics.$)$

Then, we have the following$:$
\smallskip

$(1)$ Assume $\theta$ is $d$-closed. Then, the $m$-th Narasimhan-Simha metric $H_m$ satisfies
$$
\ai\Theta_{H_m} \geq^{L^2}_{\loc \Nak} \theta \otimes \Id \text{ on } Y \text{ in the sense of Definition } \ref{def-dbarNak-sheaf}.
$$

$(2)$ If $X$ is K\"ahler and $\theta = 0$, then the $m$-th Narasimhan-Simha metric $H_m$ satisfies
$$
\ai\Theta_{H_m} \geq^{L^2}_{\glo \Nak} 0 \text{ on } Y \text{ in the sense of Definition } \ref{def-dbarNak-sheaf}.
$$

\medskip

$(3)$ Assume $\theta$ is $d$-closed and $\mathcal{I}(h|_{X_y}^{\frac{1}{m}}) =\mathcal{O}_{X_{y}}$ 
holds for a general fiber $X_{y}$ of $f \colon X \to Y$ $($which is satisfied for a sufficiently large  integer $m$$)$. Then, the natural inclusion
$$
\mathcal{E} = f_{*}(\mathcal{O}_{X}(m K_{X/Y} + L) \otimes \mathcal{I}( B_{m}^{-\frac{m-1}{m}} \cdot h^{\frac{1}{m}})) \hookrightarrow f_{*}(\mathcal{O}_{X}(mK_{X/Y} + L))
$$
is generically isomorphic, and the singular Hermitian metric $G_m$ on $f_{*}(\mathcal{O}_{X}(mK_{X/Y} +L) )$ induced by this inclusion and $H_m$ satisfies
$$
\ai\Theta_{G_m} \geq^{L^2}_{\loc \Nak} \theta \otimes \Id \text{ on } Y \text{ in the sense of Definition } \ref{def-dbarNak-sheaf}.
$$

\end{cor}

\begin{proof}
We will first prove conclusion (2). Consider the following decomposition: 
\begin{equation}\label{decom}
mK_{X/Y} + L = K_{X/Y} + \frac{m-1}{m} (mK_{X/Y}+L) + \frac{1}{m}L,
\end{equation}
and assign the singular Hermitian metrics $B_{m}^{-1}$ and $h$ to $(mK_{X/Y}+L)$ and $L$, respectively. 
Here, $B_{m}^{-1}$ denotes the $m$-th Bergman metric on $mK_{X/Y}+L$ defined by $(L,h)$.
Since the curvature of $B_{m}^{-1}$ is semi-positive by \cite[3.1]{PT18}, 
the curvature of $B_{m}^{-(m-1)/m} \cdot h^{1/m}$ is also semi-positive.
Thus, conclusion (2) follows directly from Theorem \ref{thm-main}.

We now prove conclusion (1) using the argument outlined above. 
Take a locally defined smooth function $\varphi$ such that $\theta = \deldel \varphi$ 
and consider the singular Hermitian metric $h' := h e^{f^{*}\varphi}$ on $L$. 
Note that $\sqrt{-1}\Theta_{h'} \geq 0$ holds. 
Consider the $m$-th Narasimhan-Simha metric $H_m'$ defined by ${B'}_{m}^{-1}$ and $h'$ as in the firts part, 
where ${B'}_{m}^{-1}$ is defined by $(L, h')$. 
The first part shows that $H_m'$ satisfies
\[
\ai\Theta_{H_m'} \geq^{L^2}_{\loc \Nak} 0 \text{ on } Y \text{ in the sense of Definition } \ref{def-dbarNak-sheaf}.
\]
On the other hand, by the definitions of ${B'}_{m}^{-1}$ and $h'$, 
we can see that 
\[
{B'}_{m}^{-\frac{m-1}{m}} \cdot {h'}^{\frac{1}{m}} = {B}_{m}^{-\frac{m-1}{m}} \cdot {h}^{\frac{1}{m}} e^{f^{*}\varphi}.
\]
This shows that $H_m' = H_m e^{f^{*}\varphi}$, completing the proof of conclusion (1).

We finally prove conclusion (3). 
Given the assumption about $\mathcal{I}(h^{\frac{1}{m}}|_{X_y})$, 
the natural inclusion 
$$
\mathcal{E} = f_{*}(\mathcal{O}_{X}(m K_{X/Y} + L) \otimes \mathcal{I}( B_{m}^{-\frac{m-1}{m}} \cdot h^{\frac{1}{m}})) \hookrightarrow f_{*}(\mathcal{O}_{X}(mK_{X/Y} + L)) =:\mathcal{Q} 
$$
is isomorphic over some Zariski open subset (see \cite[5.1]{PT18}).
The results in \cite{PT18, HPS18} show that the induced singular Hermitian metric $G_m$ on $\mathcal{Q}$ is Griffiths $\theta$-positive. 
In particular, the metric $G_m$ satisfies the assumption of Lemma \ref{lemm-L2} 
(i.e., the norm $|t|_{G_m^{*}}$ is locally bounded for every smooth section $t$ of $\mathcal{Q}^*$). 

We will check that the restriction of $(\mathcal{Q}, G_m)$ to $Y_\mathcal{Q}$  
satisfies Definition \ref{def-dbarNak} (see Definition \ref{def-dbarNak-sheaf}). 
Let $U$ be a sufficiently small open ball centered at a point $y \in Y$ 
and $\Omega \subset U \cap Y_\mathcal{Q}$ be a Stein coordinate with data $\psi$, $q$, and $g$ as in Definition \ref{def-dbarNak}.   
Let $Z \subset Y$ be a closed analytic subset such that 
the above inclusion is an isomorphism and both $\mathcal{E}$ and $\mathcal{Q}$ are locally free 
on $Y \setminus Z$. 
Consider a hypersurface $H$ such that $\Omega \setminus H$ is Stein and $Z \subset H$. 
Note that $g$ (which is a $\mathcal{Q}$-valued form) can be identified with 
a $\mathcal{E}$-valued form $g'$ on $\Omega \setminus Z$. 
Under this identification, by conclusion (1), we can find a $\mathcal{E}$-valued form $u'$  
satisfying $\Dbar g' = u'$ on $\Omega \setminus H$ and satisfying the $L^2$-estimate in Definition \ref{def-dbarNak}. 
Note that the $L^2$-estimate is the same whether considered over $\Omega$ or $\Omega \setminus H$ 
since $H$ is of Lebesgue measure zero. 
Then, using the identification of $\mathcal{E}$ and $\mathcal{Q}$ on $\Omega \setminus H$ again,
the solution $u'$ can be considered a $\mathcal{Q}$-valued form $u$ on $\Omega \setminus Z$
satisfying $\Dbar g = u$ on $\Omega \setminus H$. 
Lemma \ref{lemm-L2} extends the equation $\Dbar g = u$ on $\Omega \setminus H$ to $\Omega$, 
thereby completing the proof.
\end{proof}

\appendix
\section{}\label{appendix}

In this section, we summarize a standard approximation method for almost everywhere (a.e.) Griffiths semi-negative singular Hermitian metrics, as described in \cite[Definition 2.2.2]{PT18}. This result can be seen as a generalization of Berndtsson-Paun's regularization theorem for singular Hermitian metrics with Griffiths semi-negative curvature. As an application of this result, we obtain an extension theorem for Griffiths semi-negative singular Hermitian metrics. While these results may be familiar to experts, we provide detailed explanations for the convenience of readers.

%As a byproduct of the main research, we establish a standard approximation method for a.e. Griffiths semi-negative singular Hermitian metrics (see \cite[Definition\,2.2.2]{PT18}).
%This result can be seen as a generalization of Berndtsson--Paun's regularization theorem for singular Hermitian metrics with Griffiths semi-negative curvature. 
%As an application of this result, we obtain an extension theorem for Griffiths semi-negative singular Hermitian metrics, which may be known for some experts, but our approach is based on a different perspective than before. 
%For readers who are interested in these topics, we will note them down. 

\begin{lemm}[{cf.\,\cite[Proposition\,3.1]{BP08},\, \cite[Remark 2.2.3]{PT18},\, \cite[Proof of Theorem\,5.10]{Wat24}}]\label{lem-aeGrif}
    Let $h$ be an a.e. Griffiths semi-negative singular Hermitian metric on the trivial bundle $U\times \C^r$ on a domain $U\subset \C^n$ and $h_{i\bar{j}}$ denote $h(e_i, e_j)$, 
    where $(e_1,\ldots,e_r)$ is the standard frame of $U\times \C^r$. 
    For a standard approximate identity $\rho_\e$, we define a sequence of smooth Hermitian metrics $\{ h_\e\}_{\e>0}$ by the convolution 
    \[
    h_{\e,i\bar{j}}=h_{i\bar{j}}* \rho_e. 
    \]
    Then, the following hold: 
    \begin{enumerate}
        \item $h_\e$ is a Griffiths semi-negative smooth Hermitian metric defined on $U_\e=\{ x\in U \mid \mathbb{B}(x;\e)\subset U\}$; 
        \item there exists a Griffiths semi-negative singular Hermitian metric $H$ on U such that $h_\e \searrow H$ pointwise as $\e \searrow 0$; 
        \item $H=h$ almost everywhere. 
    \end{enumerate}
\end{lemm}

\begin{proof}
    For any constant section $c$, we get $|c|^2_h*\rho_\e=|c|^2_{h_\e}$ and $\deldel |c|^2_h\geq 0$, which means that $h_\e$ is decreasing as $\e\searrow 0$. To prove $h_\e$ to be Griffiths semi-negative, we take an arbitrary open subset $V\Subset U_\e$ and a local holomorphic section $u=(u_1,\ldots,u_r)$ on $V$. 
    It is enough to show that $\Delta_\xi(|u|^2_{h_\e})\geq0$ for any $\xi=(\xi_1,\ldots,\xi_n)\in \C^n$, where $\Delta_\xi=\sum \frac{\partial^2}{\partial z_j\partial \bar{z}_k}\xi_j\bar{\xi}_k$. 
    For any test function $\phi\in\cal{D}(V)_{\geq0}$, by Fubini's theorem we have
    \begin{align*}
        \int_V\Delta_\xi(|u|^2_{h_\e})\phi&=\int_V|u|^2_{h_\e}(z)\Delta_\xi(\phi)(z)dV_z\\
        %&=\int_V \sum_{i,j} h_{\varepsilon,i\bar{j}}(z)u_i(z)\overline{u}_j(z)\Delta_\xi(\phi)(z)dV_z\\
        &=\int_V \sum_{i,j}\biggl(\int_{\mathrm{supp}\,\rho_\varepsilon} h_{\varepsilon,i\bar{j}}(z-w)\rho_\varepsilon(w)dV_w\biggr)u_i(z)\overline{u}_j(z)\Delta_\xi(\phi)(z)dV_z\\
        %&=\int_{\mathrm{supp}\,\rho_\varepsilon} \sum_{i,j}\biggl(\int_V h_{\varepsilon,i\bar{j}}(z-w)u_i(z)\overline{u}_j(z)\Delta_\xi(\phi)(z)dV_z\biggr)\rho_\varepsilon(w)dV_w\\
        &=\int_{\mathrm{supp}\,\rho_\varepsilon}\biggl(\int_V |u|^2_{h^w}(z)\Delta_\xi(\phi)(z)dV_z\biggr)\rho_\varepsilon(w)dV_w\\
        &\geq0,
    \end{align*}
    where $h^w(z):=h(z-w)$. Since it is only a parallel transport, \( h^w \) is also a.e. Griffiths semi-negative, i.e. $\Delta_\xi(|u|^2_{h^w})\phi\geq0$.

    Define the singular Hermitian metric $H$ by $H(z):=\lim_{\varepsilon\to+\infty}h_\varepsilon(z)$. %, i.e. $h^*_G$ is the limit of convergence pointwise.
    In detail, $H$ is defined as follows. Let $H(z):=(H_{i\bar{j}}(z))$ be the local matrix representation for $(e_1,\ldots,e_r)$. 
    For constant sections $v_1={}^t(1,0,\cdots,0),\cdots,v_r$ $={}^t(0,\cdots,0,1)$, define $H_{jj}$ as the decreasing limit of $|v_j|^2_{h_\varepsilon}$ converging pointwise, i.e. $H_{jj}(z):=\lim_{\varepsilon\to 0}|v_j|^2_{h_\varepsilon}(z)$, then $H_{jj}$ is a function with plurisubharmonicity and coincides with $h_{jj}$ a.e. % almost everywhere.
    For constant vectors $v={}^t(1,1,\cdots,0)$ and $v'={}^t(1,i,\cdots,0)$, we obtain 
    \begin{align*}
        \lim_{\varepsilon\to+\infty}|v|^2_{H_\varepsilon}&:=|v|^2_{H}=H_{11}+H_{22}+2\mathrm{Re}H_{12},\\ 
        \lim_{\varepsilon\to+\infty}|v'|^2_{H_\varepsilon}&:=|v'|^2_{H}=H_{11}+H_{22}+2\mathrm{Im}H_{12}.
    \end{align*} 
    From this, $H_{12}$ can be defined as a function and coincides with $h_{12}$ a.e. By finitely repeating this process, we can define $H$ that coincides with $h$ a.e.
    Since $h_\e$ is decreasing as $\e \searrow 0$, it follows that $\log |u|^2_H= \lim_{\varepsilon\to 0}\log|u|^2_{h_\varepsilon}$ is plurisubharmonic for any local holomorphic section $u$, which means that $H$ is Griffiths semi-negative. 
\end{proof}

This lemma naturally leads to the following result.

\begin{lemm}\label{lem-aeGrifglobal}
    Let $X$ be a complex manifold, $E\to X$ be a holomorphic vector bundle on $X$ and $h$ be an a.e. Griffiths semi-negative singular Hermitian metric on $E$. Then, there exists a Griffiths semi-negative singular Hermitian metric $H$ on $E$ satisfying $H=h$ almost everywhere.
\end{lemm}

\begin{proof}
    Take a sufficiently fine open covering $\{ U_\alpha\}_{\alpha\in \Lambda}$. 
    Then we regard $h$ as a collection of singular Hermitian metrics $\{ h_\alpha\}_{\alpha\in \Lambda}$ defined on $U_\alpha$ satisfying ${}^t\overline{g_{\alpha\beta}}h_\alpha g_{\alpha\beta}=h_\beta$ a.e. on $U_\alpha\cap U_\beta$, where each $h_\alpha$ is a.e. Griffiths semi-negative and $\{g_{\alpha\beta}\}_{\alpha,\beta\in\Lambda}$ is a collection of transition functions of $E$.
    By using Lemma \ref{lem-aeGrif}, we can obtain a Griffiths semi-negative singular Hermitian metric $H_\alpha$ on $U_\alpha$. Since $\{ H_\alpha\}_{\alpha\in \Lambda}$ also satisfies  ${}^t\overline{g_{\alpha\beta}}H_\alpha g_{\alpha\beta}=H_\beta$ a.e. on $U_\alpha\cap U_\beta$, the collection $\{ H_\alpha\}_{\alpha\in \Lambda}$ defines a singular Hermitian metric $H$ on $E$. 
    Since each $g_{\alpha\beta}$ is holomorphic, the function $\log|s_\beta|^2_H=\log|s_\beta|^2_{H_\beta}=\log|g_{\alpha\beta}s_\beta|^2_{H_\alpha}$ is plurisubharmonic on $U_\alpha\cap U_\beta$ for any $s_\beta\in H^0(U_\beta,E)$, where $g_{\alpha\beta}s_\beta\in H^0(U_\alpha\cap U_\beta,E)$. Then \( H \) is globally Griffiths semi-negative.
\end{proof}

At the end of this section, we show the aforementioned extension result, which demonstrates that Griffiths semi-negative singular Hermitian metrics behave like plurisubharmonic functions.

\begin{lemm}\label{lemm-Grifkakucho}
    Let $U\subset \C^n$ be a domain, $Z\subsetneq U$ be a $($closed$)$ analytic subset 
    and $E$ be a vector bundle on $U$. 
    Let $h$ be a Griffiths semi-negative singular Hermitian metric defined on $E|_{U\setminus Z}$. 
    Assume one of the following conditions: 
    \begin{enumerate}
        \item there exists an upper semi-continuous singular Hermitian metric $h'$ on E satisfying $h'|_{U\setminus Z}=h$; 
        \item $\mathrm{codim}\, Z \geq 2$. 
    \end{enumerate}
    Then, there exists a Griffiths semi-negative singular Hermitian metric $H$ on $E$ such that $H|_{U\setminus Z}=h$. 
\end{lemm}

\begin{proof}
Regarding $h$ as a metric on $E$ by the trivial extension, we can say that $h$ is a.e. Griffiths semi-negative thanks to the assumption. 
Then, Lemma \ref{lem-aeGrif} implies that there exists a Griffiths semi-negative singular Hermitian metric $H$ on $E$ such that $H=h$ almost everywhere. 
Since both $H$ and $h$ are Griffiths semi-negative on $U\setminus Z$, it follows that $H=h$ on $U\setminus Z$. 
\end{proof}


\begin{thebibliography}{n}

\bibitem[BP08]{BP08} B.~Berndtsson, M.~P{\u{a}}un, \textit{Bergman kernels and the pseudoeffectivity of relative canonical bundles}, Duke Math. J. \textbf{145} (2008), no.~2, 341--378.

%\bibitem[Ber98]{Ber98} B.~Berndtsson, \textit{Prekopa's theorem and Kiselman's minimum principle for plurisubharmonic functions}, Math. Ann. \textbf{312} (1998), 785--792.

\bibitem[Ber09]{Ber09} B.~Berndtsson, \textit{Curvature of vector bundles associated to holomorphic fibrations}, Ann. of Math. \textbf{169} (2009), no.~2, 531--560.


\bibitem[BL16]{BL16} B.~Berndtsson, L.~Lempert, 
\textit{ A proof of the Ohsawa-Takegoshi theorem with sharp estimates}, 
Math. Soc. Japan, {\bf{68}} (4) (2016), 1461--1472.


%\bibitem[deC98]{deC98} M.~A.~A. de~Cataldo, \textit{Singular {H}ermitian metrics on vector bundles}, J. Reine Angew. Math. \textbf{502} (1998), 93--122.

\bibitem[Dem82]{Dem82} J.-P.~Demailly, \textit{Estimations $L^2$ pour l'operateur $\overline{\partial} $ d'un fibre vectoriel holomorphe semipositif au dessus d'une variete kahlerienne complete}, Ann. Sci. {\'E}cole Norm. Sup. \textbf{15} (1982), 457--511.

\bibitem[Dem92]{Dem92} J.-P.~Demailly, \textit{Regularization of closed positive currents and intersection theory}, J. Algebraic Geom. \textbf{1} (1992), no.~3, 361--409.

\bibitem[Dem12]{Dem12} J.-P.~Demailly, \textit{Analytic methods in algebraic geometry}, Surveys of Modern Mathematics, vol.~1, International Press, Somerville, MA; Higher Education Press, Beijing, 2012.

\bibitem[Dem-book]{Dem-book} J.-P.~Demailly, \textit{Complex analytic and differential geometry}, https://www-fourier.ujf-grenoble.fr/$\sim$demailly/manuscripts/agbook.pdf.

\bibitem[DNWZ23]{DNWZ23} F.~Deng, J.~Ning, Z.~Wang, X.~Zhou, \textit{Positivity of holomorphic vector bundles in terms of $L^p$-estimates for $\overline{\partial}$}, Math. Ann. \textbf{385} (2023), 575--607.

\bibitem[Fuj78]{Fuj78} T.~Fujita, \textit{On K\"ahler fiber spaces over curves}, J. Math. Soc. Japan \textbf{30} (1978), no. 4, 779--794.

\bibitem[GMY22]{GMY22} Q.~Guan, Z.~Mi, Z.~Yuan, \textit{Optimal $L^2$ extension for holomorphic vector bundles with singular hermitian metrics}, arXiv:2210.06026.

\bibitem[GMY23]{GMY23} Q.~Guan, Z.~Mi, Z.~Yuan, \textit{Boundary points, minimal $L^2$ integrals and concavity property V -- vector bundles}, J. Geom. Anal. \textbf{33}, 305 (2023).

\bibitem[Gri70]{Gri70} P.~A. Griffiths, \textit{Periods of integrals on algebraic manifolds. III. Some global differential-geometric properties of the period mapping}, Inst. Hautes \'Etudes Sci. Publ. Math. No. \textbf{38} (1970), 125--180.

%\bibitem[HI21]{HI21} G.~Hosono, T.~Inayama, \textit{A converse of H\"ormander's $L^2$-estimate and new positivity notions for vector bundles}, Sci. China Math. \textbf{64} (2021), 1745--1756.

\bibitem[HPS18]{HPS18} C.~Hacon, M.~Popa, C.~Schnell, \textit{Algebraic fiber spaces over abelian varieties: around a recent theorem by Cao and P{\v{a}}un}, Local and global methods in algebraic geometry, 143--195, Contemp. Math., \textbf{712}, Amer. Math. Soc., Providence, RI, 2018.

%\bibitem[Ina20]{Ina20} T.~Inayama, \textit{$L^2$ estimates and vanishing theorems for holomorphic vector bundles equipped with singular Hermitian metrics}, Michigan Math. J. \textbf{69} (2020), 79--96.

\bibitem[Ina22]{Ina22} T.~Inayama, \textit{Nakano positivity of singular Hermitian metrics and vanishing theorems of Demailly-Nadel-Nakano type}, Algebr. Geom. \textbf{9} (2022), 69--92.

\bibitem[IM24]{IM24} T.~Inayama, S.~Matsumura, \textit{Nakano positivity of singular Hermitian metrics: Approximations and applications}, arXiv:2402.06883.

%\bibitem[Iwa21]{Iwa21} M.~Iwai, \textit{Nadel-Nakano vanishing theorems of vector bundles with singular Hermitian metrics}, Annales de la Facult\'e des sciences de Toulouse: Math\'ematiques (6), \textbf{30} (2021), no.~1, 63--81.

\bibitem[Kaw81]{Kaw81} Y.~Kawamata, \textit{Characterization of abelian varieties}, Compositio Math. \textbf{43} (1981), no.~2, 253--276.

\bibitem[Kol86a]{Kol86a} J.~Koll\'ar, \textit{Higher direct images of dualizing sheaves. I}, Ann. of Math. (2) \textbf{123} (1986), no.~1, 11--42.

\bibitem[Kol86b]{Kol86b} J.~Koll\'ar, \textit{Higher direct images of dualizing sheaves. II}, Ann. of Math. (2) \textbf{124} (1986), no.~1, 171--202.

\bibitem[Mat16]{Mat16} S.~Matsumura, \textit{A vanishing theorem of Koll\'ar-Ohsawa type}, Mathematische Annalen, Volume 366 (2016) no.3, pp.1451--1465.

\bibitem[Mat18]{Mat18} S.~Matsumura, \textit{An injectivity theorem with multiplier ideal sheaves of singular metrics with transcendental singularities}, J. Algebraic Geom. \textbf{27} (2018), no.~2, 305--337.

\bibitem[Mat22]{Mat22} S.~Matsumura, \textit{Injectivity theorems with multiplier ideal sheaves for higher direct images under K\"ahler morphisms}, Algebr. Geom. \textbf{9} (2022), 122--158.

\bibitem[MY04]{MY04} F.~Maitani, H.~Yamaguchi, \textit{Variation of Bergman metrics on Riemann surfaces}, Math. Ann. \textbf{330} (2004), no.~3, 477--489.

\bibitem[Ohs84]{Ohs84} T.~Ohsawa, \textit{Vanishing theorems on complete K\"ahler manifolds}, Publ. Res. Inst. Math. Sci. \textbf{20} (1984), no.~1, 21--38.

\bibitem[PT18]{PT18} M.~P{\u{a}}un, S.~Takayama, \textit{Positivity of twisted relative pluricanonical divisors and their direct images}, J. Algebraic Geom. \textbf{27} (2018), 211--272.

\bibitem[Rau15]{Rau15} H.~Raufi, \textit{Singular hermitian metrics on holomorphic vector bundles}, Ark. Mat. \textbf{53} (2015), no.~2, 359--382.

\bibitem[Vie83]{Vie83} E.~Viehweg, \textit{Weak positivity and the additivity of the Kodaira dimension for certain fibre spaces}, Algebraic varieties and analytic varieties (Tokyo, 1981), 329--353, Adv. Stud. Pure Math., \textbf{1}, North-Holland, Amsterdam, 1983.

\bibitem[Wat23]{Wat23} Y.~Watanabe, \textit{Dual Nakano positivity and singular Nakano positivity of direct image sheaves}, accepted for publication in Nagoya Math. J.  arXiv:2302.09398.

\bibitem[Wat24]{Wat24} Y.~Watanabe, \textit{$\omega$-trace and Griffiths positivity for singular Hermitian metrics}, arXiv:2402.06658.

\bibitem[ZZ18]{ZZ18} X.~Zhou, L.~Zhu, \textit{An optimal $L^2$ extension theorem on weakly pseudo-convex K\"ahler manifolds}, J. Differential Geom. \textbf{110} (2018), 135--186.




\end{thebibliography}
\end{document}